\documentclass[12pt]{amsart}
\newtheorem{thm}{Theorem}
\newtheorem{lemma}[thm]{Lemma}
\newtheorem{cor}[thm]{Corollary}
\newcommand{\tr}{\operatorname{tr}}
\newcommand{\rk}{\operatorname{rank}}
\begin{document}
\title{On Functions with a Conjugate}
\author{Paul Baird}
\address{\hskip-\parindent
Laboratoire de Math\'ematiques de Bretagne Atlantique UMR 6205\\
Universit\'e de Bretagne Occidentale, 
29238 Brest Cedex 3\\
France}
\email{Paul.Baird@univ-brest.fr}
\author{Michael Eastwood}
\address{\hskip-\parindent
Mathematical Sciences Institute\\
Australian National University,\newline
ACT 0200, Australia}
\email{meastwoo@member.ams.org}
\thanks{The first author is grateful for support provided by the Australian
Research Council and to the Mathematical Sciences Institute at the Australian
National University; part of this work was carried out under the award of a
\emph{d\'el\'egation aupr\`es du CNRS}. The second author is a Federation
Fellow of the Australian Research Council and is grateful to the D\'epartement
de Math\'ematiques \'a l'Universit\'e de Bretagne Occidentale for support and
hospitality while working on this article. }
\begin{abstract}
Harmonic functions of two variables are exactly those that admit a conjugate,
namely a function whose gradient has the same length and is everywhere
orthogonal to the gradient of the original function. We show that there are
also partial differential equations controlling the functions of three
variables that admit a conjugate.
\end{abstract}
\maketitle
\section{Introduction}
A pair of smooth real-valued functions $f$ and $g$ on a Riemannian manifold~$M$
are said to be conjugate if and only if
\begin{equation}\label{definition}
\|\nabla f\|=\|\nabla g\|\quad\mbox{and}\quad
\langle\nabla f,\nabla g\rangle=0.
\end{equation}
In this article, we shall address the following question. When does a given
smooth function $f:M\to{\mathbb{R}}$ admit a conjugate function? When $M$ is
2-dimensional the pair of functions $(f,g):M\to{\mathbb{R}}^2$ is mutually
conjugate if and only if the mapping $(f,g)$ is conformal away from isolated
points where its differential vanishes. It is well-known that, in this case,
$f$ must be harmonic and, conversely, a harmonic function locally always admits
a conjugate, unique up to an additive constant. When $M$ is of higher
dimension, then the pair $(f,g):M\to{\mathbb{R}}^2$ is said to be
semiconformal. As discussed in~\cite{thebook}, semiconformality is one of the
two conditions that $(f,g)$ be a harmonic morphism. In fact, if
$M={\mathbb{R}}^n$ and both $f$ and $g$ are polynomial, then it is the only
condition~\cite{abb}. In this article, we shall be concerned with $f$ defined
on an open subset in~${\mathbb{R}}^3$. We extend our earlier work~\cite{be} in
which we derived some necessary conditions on $f$ in order that it admit a
conjugate under a non-degeneracy condition, to now obtain necessary and sufficient conditions in all cases.

An example of a pair of conjugate functions in three
variables is
$$f=x_2\frac{x_1{}^2+x_2{}^2+x_3{}^2}{x_2{}^2+x_3{}^2}\qquad
g=x_3\frac{x_1{}^2+x_2{}^2+x_3{}^2}{x_2{}^2+x_3{}^2}.$$
The Hopf mapping $S^3\to S^2$ viewed in stereographic co\"ordinates 
$$f=\frac{(1-\|x\|^2)x_2+2x_1x_3}{x_2{}^2+x_3{}^2}\qquad
g=\frac{(1-\|x\|^2)x_3-2x_1x_2}{x_2{}^2+x_3{}^2}$$
provides another good example. In these two cases, the pair $(f,g)$ enjoys an
evident symmetry with respect to rotation about the \mbox{$x_1$-axis}. This is
not usual, as is illustrated by the following example:--
$$f=\log\sqrt{x_1{}^2+x_2{}^2+x_3{}^2}\qquad
g=\arccos \frac{x_1}{\sqrt{x_1{}^2+x_2{}^2+x_3{}^3}}.$$
In all three examples, the pair $(f,g)$ is smooth away from the $x_1$-axis.

We shall frequently need to manipulate tensors and for this purpose, we use 
Penrose's abstract index notation~\cite{OT}. We shall write 
$$f_i=\nabla_if\qquad f_{ij}= \nabla_i\nabla_jf \qquad\mbox{et cetera},$$
where $\nabla_i$ is the flat connection on ${\mathbb{R}}^n$ or, more generally,
the metric connection on a Riemannian manifold. Also, let us `raise and lower'
indices with the metric $\delta_{ij}$ in the usual fashion and write a repeated index to
denote the invariant contraction over that index. Thus, $f^i{}_i=\Delta f$ is
the Laplacian and $f^ig_i=\langle\nabla f,\nabla g\rangle$. We shall use round
and square brackets to denote symmetrising and skewing over the indices they
enclose. For example, $\phi_{(ij)k}=\frac12\phi_{ijk}+\frac12\phi_{jik}$ and
$\nabla_{[i}\phi_{j]}$ is the exterior derivative of a $1$-form~$\phi_i$.

In order to find necessary and sufficient conditions for a function $f$ defined on an open set of ${\mathbb{R}}^3$ to admit a conjugate, we begin by constructing conformal invariants that reflect geometric constraints that derive from \eqref{definition} and its derivatives. 

A conformal differential invariant is a polynomial in the derivatives of $f$ as well as the inverse (Euclidean) metric, that transforms by scaling under the action of the M\"obius group on ${\mathbb{R}}^3 \cup \{ \infty\}$ (the amount of scaling being called the \emph{weight} of the invariant: for details see Appendix \ref{conformalinvariants}).  An elementary conformal invariant is the first order one $J:= f^if_i$ of weight $-2$.  We shall require invariants up to third order. In Appendix \ref{conformalinvariants} we give a more thorough treatment of conformal invariants and derive a list of those that we require; these will be labelled with  uppercase Roman letters.   

Higher order conformal invariants may be built from lower order ones by using simple rules.  For example, if $\phi_i$ is a conformally invariant $1$-form of weight $-1$, then the trace $\nabla^i\phi_i$ is conformally invariant.  Applying this procedure to the $1$-form $\sqrt{J}f_i$ yields $Z/\sqrt{J}$, where, up to a multiple, the operator $Z$ is the $3$-Laplacian, a well-know conformal invariant in dimension $3$.     The trace-free part of $\nabla_{(i}\phi_{j)}$ is invariant whenever $\phi_j$ has weight $2$.  On applying this construction to $J^{-1}f_i$ yields an invariant $\psi_{ij}$ from which we deduce another invariant $X$ via the formula:
$$
\psi^{ij}\psi_{ij} = \textstyle{\frac23} Z^2 - JX\,.
$$

The invariant $X$ plays a fundamental role in our characterization.  Its explicit expression is given in \S\ref{sec:nc} below.  A necessary condition that $f$ admit a conjugate is that $X\leq 0$ (Theorem \ref{one}).  In what we refer to as the \emph{generic case} $X<0$, there are exactly four distinct vectors (two up to sign) called \emph{conjugate directions}, which potentially may be the gradient of a conjugate function.  When $X=0$ there are either exactly two conjugate directions, so up to sign any conjugate must be unique, or infinitely many; these two cases are distinguished by another conformal invariant derived from $X$ and $Z$, which we call $Y$.  By normalising co\"ordinates, we explain the geometric interpretation of these conditions.

The next step is to understand when a conjugate direction $\omega_i$ is integrable and so is the gradient of a function.  In \S\ref{sec:gc} we show that in the generic case, integrability is equivalent to the vanishing of two polynomial expressions in $\omega_i$ and the derivatives up to third order of $f$ (Theorem \ref{thm:integrability1}).  Our objective is then to eliminate $\omega_i$ to obtain conditions involving just derivatives of $f$.  However, a difficulty arises in that we only have explicit expressions for quadratic terms in $\omega_i$.  Thus, instead of trying to determine whether a specific conjugate direction is integrable, we ask rather that one or the other be integrable without specifying which.  This leads to a set of three equations involving just quadratic terms in $\omega_i$ (Theorem \ref{mainthm}). In \S\ref{sec:resolution}, we show how to elimiate $\omega_i$ in a normalized co\"ordinate system to give three third order differential equations in $f$.  Each equation is a conformally invariant homogeneous expression in the derivatives of $f$ with a certain weight and degree.  To write these down in terms of conformal invariants, we explore combinations of invariants that have the same weight and degree and use ad hoc methods to equate terms.
An invariant derivation without recourse to normal co\"ordinates is given in Appendix \ref{sec:elim-om}.   

In \S\ref{sec:sc} we deal with special cases, the first of which concerns functions that admit a unique conjugate direction (up to sign).  In terms of conformal invariants, these are characterized by the conditions $X=0$ and $Y\neq 0$.  The analysis proceeds in a similar way to the generic case, except that now the characterization requires just two third order equations, made explicit in Corollary  \ref{cor:unique}.  The next special case concerns functions that admit infinitely many conjugates, characterized by $X=Y=0$. Now, $J^{-1}f_j$ is a conformal Killing field, all of which can be written down explicitly, as detailed in Appendix \ref{normalised_conf_Killing}.  This enables us to write down all conjugate pairs in this case.  The final special case discusses functions of two variables that admit a conjugate (in ${\mathbb{R}}^3$).  

Examples are discussed in \S\ref{sec:examples}.  For the case of spherical symmetry, up to scaling and addition of a constant, $\log ||x||$ is the unique function that admits a conjugate, in fact infinitely many.  If $f$ is assumed to have cylindrical symmetry, then the corresponding examples give a nice illustration of the generic case.  For a conjugate pair $(f,g)$, fibres of the associated map into ${\mathbb{R}}^2$ are helices which wind around concentric cylinders; right-handed screw or left-handed screw now corresponds to the two choices of conjugate.  Finally, in \S\ref{sec:invariantsconjugate}, for a function $f$ that admits a conjugate $g$, we discuss how the conformal invariants $X(g)$ and $Z(g)$ of $g$ relate to those of $f$.  This enables us to give a characterization of $3$-harmonic conjugate pairs.

\section{A necessary condition} \label{sec:nc}
\begin{thm}\label{one} Let $M$ be an $3$-dimensional Riemannian manifold and
$f:M\to{\mathbb{R}}$ a smooth function. In order to admit a conjugate, $f$ must
satisfy
the differential inequality
\begin{equation}\label{differentialinequality}
X:= 2f_i{}^jf_jf^{ik}f_k-f^if_if^{jk}f_{jk}+f^if_i(f^j{}_j)^2\leq 0.
\end{equation}
\end{thm}
\begin{proof} A proof of this theorem was given in~\cite{be}. In fact, a
version was proved there valid in any dimension. Here we give a more efficient
proof only valid in three dimensions. However, this proof will allow us to draw
additional and useful conclusions. In addition, the method of proof (in Lemma \ref{magic}) will
provide a good illustration of the normalisation techniques occurring
throughout the rest of this article.

If $f$ is to admit a conjugate, then there must be a closed
$1$-form $\omega_j$ so that
\begin{equation}\label{omega}
f^j\omega_j=0\quad\mbox{and}\quad\omega^j\omega_j=f^jf_j.
\end{equation}
Indeed, (\ref{definition}) implies that we may find an $\omega_j$ that is
exact. We shall show that the inequality (\ref{differentialinequality}) is
necessary in order to find a closed $\omega_j$ satisfying~(\ref{omega}). To
proceed, let us differentiate the equations~(\ref{omega}) with~$\nabla^i$. We
obtain 
\begin{equation}\label{diff1}
f^{ij}\omega_j+\omega^{ij}f_j=0\quad\mbox{and}\quad
\omega^{ij}\omega_j=f^{ij}f_j.\end{equation}
Since we are supposing that $\omega_{ij}=\nabla_i\omega_j$ is symmetric
we may transvect the second of these with $f_i$ and use the first to
eliminate~$\omega^{ij}f_i$. This gives
$$f^{ij}\omega_i\omega_j+f^{ij}f_if_j=0.$$
We now have the following equations
\begin{equation}\label{firstthree}
f^i\omega_i=0\qquad\omega^i\omega_i=f^if_i\qquad
f^{ij}\omega_i\omega_j+f^{ij}f_if_j=0
\end{equation}
and we claim it is a matter of algebra to show that the
inequality~(\ref{differentialinequality}) must hold if there is to be a
solution~$\omega_i$. This is detailed in the following Lemma, which we state independently for future use.  Notice that if $\omega_i$ is real then so is $T_{ijk}$ in which case $T_{ijk}T^{ijk}\geq 0$.   
\end{proof}
\begin{lemma}\label{magic}
If $f_{ij}$ is a $3\times 3$ symmetric matrix and $f_i$ is a $3$-vector, then
\begin{equation}\label{criticalidentity}
(f^if_i)X
+12T_{ijk}T^{ijk}=0\end{equation}
where
\begin{equation}\label{Tijk}
T_{ijk}=f_{[i}\omega_jf_{k]\ell}\omega^\ell\end{equation}
and $\omega_i$ is any solution, real or complex, of the
equations~{\rm(\ref{firstthree})}.
\end{lemma}
\begin{proof}
If $f_i=0$ then the conclusion is trivial. Otherwise, let us choose
co\"ordinates so that $f_1=f_2=0$. We may also orthogonally diagonalise the
quadratic form $f_{ij}$ restricted to the plane orthogonal to~$f_i$. In other
words, we may further change co\"ordinates to arrange that $f_{12}=0$. Having
made these choices, the quantity $X$ becomes, after a short calculation,
\begin{equation}\label{Xnormform}X=2(f_3)^2(f_{11}+f_{33})(f_{22}+f_{33}).
\end{equation}
Another short calculation yields
\begin{equation}\label{Tsquarednormform}
\textstyle T_{ijk}T^{ijk}
=\frac16(f_{22}-f_{11})^2\omega_1{}^2\omega_2{}^2
\end{equation}
whilst the equations (\ref{firstthree}) become
\begin{equation}\label{become}
\omega_3=0\qquad\omega_1{}^2+\omega_2{}^2=f_3{}^2\qquad
f_{11}\omega_1{}^2+f_{22}\omega_2{}^2+f_{33}f_3{}^2=0
\end{equation}
the second two of which may be written as
\begin{equation}\label{matrixform}
\left[\begin{array}{cc}1&1\\
f_{11}+f_{33}&f_{22}+f_{33}\end{array}\right]
\left[\begin{array}c\omega_1{}^2\\ \omega_2{}^2
\end{array}\right]=
\left[\begin{array}cf_3{}^2\\ 0
\end{array}\right].
\end{equation}
Now there are two cases. If $f_{11}=f_{22}$, then (\ref{Tsquarednormform})
implies that $T_{ijk}T^{ijk}=0$. But (\ref{matrixform}) implies that
$f_{11}+f_{33}=0$ and then (\ref{Xnormform}) shows that $X=0$ and
(\ref{criticalidentity}) reduces to $0=0$. On the other hand, if
$f_{11}\not=f_{22}$, then we may use (\ref{matrixform}) to 
solve~(\ref{become}), obtaining
\begin{equation}\label{solution}
\omega_1{}^2=f_3{}^2\frac{f_{22}+f_{33}}{f_{22}-f_{11}}
\quad\mbox{and}\quad
\omega_2{}^2=f_3{}^2\frac{f_{11}+f_{33}}{f_{11}-f_{22}}.
\end{equation}
and compute
$$12T_{ijk}T^{ijk}=
2(f_{22}-f_{11})^2\omega_1{}^2\omega_2{}^2=
-2f_3{}^4(f_{11}+f_{33})(f_{22}+f_{33}).$$
A comparison with (\ref{Xnormform}) immediately
yields~(\ref{criticalidentity}), as required.
\end{proof}

{From} now on we shall suppose that $f_i$ is non-zero (at a particular point 
and hence nearby as well). In case that $f$ admit a conjugate, it is then clear 
from (\ref{definition}) that the pair $(f,g)$ is a submersion (near the point 
in question). The nature of the singularities of a semiconformal mapping is 
not known in general~\cite{jarcs2}. 

Notice that it follows from the proof of this lemma that the equations
(\ref{firstthree}) always have solutions if we allow~$\omega_{i}$ to be complex
and generically (in fact, precisely when $X\not=0$) there are four solutions.
Alternatively, this is geometrically clear: the first equation restricts
matters to a plane wherein the second and third equations describe planar
quadrics.

Perhaps our proof of Lemma~\ref{magic} seems bizarre but, in fact, we have used
a familiar technique. The Cayley-Hamilton Theorem, for example, is often
proved, even for real matrices, by employing Jordan canonical form over the
complex numbers. Not only that, but Lemma~\ref{magic} can be proved without
normalisation by means of the Cayley-Hamilton Theorem applied to $f_{ij}$
restricted, as a symmetric form, to the plane orthogonal to~$f_i$ (the details 
of this proof being left to the reader). Another
proof avoiding normalisation may be obtained by expanding the identity
$0=f_{[i}\omega_jf_k{}^kf_{\ell]}{}^{\ell}$. In fact, it is a consequence of
Weyl's Second Fundamental Theorem of Invariant Theory~\cite{w} that
dimension-dependent identities must arise by `skewing over too many indices'.
To use normalisation as we have done, however, is a simple enough method that
we shall employ throughout this article. 

The quantities occurring in the proof of Lemma~\ref{magic} suggest other 
combinations of derivatives with geometric significance. The operator
\begin{equation}\label{3Lap}
f\mapsto Z\equiv f^{ij}f_if_j+f^if_if^j{}_j,\end{equation} for example is, up
to a multiple, the well-known $3$-Laplacian \cite{dick,tom} and in normal
co\"ordinates
\begin{equation}\label{nor}f_1=f_2=f_{12}=0\end{equation}
at a point becomes
\begin{equation}\label{thisisZ}Z=f_3{}^2(f_{11}+f_{22}+2f_{33}).\end{equation}
Also, the quantity $J\equiv f^if_i$ is $f_3{}^2$. Therefore, 
from~(\ref{Xnormform}),
\begin{equation}\label{thisisY}
\begin{array}{rcl}Y&\equiv&Z^2-2JX\\[3pt]
&=&f_3{}^4(f_{11}+f_{22}+2f_{33})^2
   -4(f_3)^4(f_{11}+f_{33})(f_{22}+f_{33})\\[3pt]
&=&f_3{}^4(f_{11}-f_{22})^2\end{array}\end{equation}
and we recognise that the vanishing of this expression when $X=0$ is exactly
the criterion discovered in the proof of Lemma~\ref{magic} for there to be
infinitely many solutions $\omega_i$ to the system~(\ref{firstthree}). In
summary, if we allow complex solutions of (\ref{firstthree}) then
\begin{equation}\begin{array}{rcl}
X\not=0&\iff&\exists\mbox{ $4$ distinct solutions}\\[3pt]
X=0\mbox{ and }Y\not=0&\iff&\exists\mbox{ $2$ distinct solutions}\\[3pt]
X=0\mbox{ and }Y=0&\iff&\exists\mbox{ $\infty$-many solutions.}
\end{array}\end{equation}
If we restrict attention to the case when (\ref{firstthree}) has real
solutions, then Lemma~\ref{magic} implies that $X\leq 0$ whence
\begin{equation}\label{solutionsintherealcase}\begin{array}{rcl}
X\not=0&\iff&\mbox{$X<0$ and $\exists\ 4$ distinct solutions}\\[3pt]
X=0\mbox{ and }Y\not=0&\iff&
             \mbox{$Y>0$ and $\exists\ 2$ distinct solutions}\\[3pt]
Y=0&\iff&\mbox{$X=0$ and $\exists\ \infty$-many solutions,}
\end{array}\end{equation}
the last two conclusions following from $Y=Z^2-2JX$ upon noting that both
terms on the right hand side are non-negative.

\section{Integrability of the conjugate direction:\\ the generic case} \label{sec:gc}
Recall that if $f$ is to admit a conjugate function near any particular point,
then there must be a solution $\omega_j$ at that point of the algebraic
equations~(\ref{firstthree}). These three equations, specifically the third
one, were derived under the assumption that $\omega_j$ extend to a closed form
near the point in question but our approach from now on is to take $\omega_j$
to be defined at a particular point by the equations (\ref{firstthree}) and ask
whether it may be extended to a smooth closed form near that point whilst
maintaining~(\ref{firstthree}). This is entirely equivalent to finding a local
conjugate for~$f$. As a matter of terminology, we shall refer to a solution
$\omega_j$ of (\ref{firstthree}) as a conjugate direction. In case that $X<0$
(at the point in question and hence nearby as well), we have just seen from
(\ref{solutionsintherealcase}) that there are four distinct solutions of
(\ref{firstthree}) for~$\omega_j$. It follows that any one of these solutions
uniquely and smoothly extends as a conjugate direction. Therefore, the only
remaining question in case $X<0$ is whether this extension is closed and we
shall refer to this as integrability. We show that integrability is equivalent
to a further two polynomial equations in $\omega^i$ and the derivatives of~$f$.

Resolution of these further equations combined with (\ref{firstthree}) will
lead to necessary and sufficient differential conditions on the function $f$ in
order that it admit a conjugate. All of this is under the assumption that
$X<0$ and we shall refer to this as the generic case. The case $X\equiv 0$
will be studied separately.

\begin{thm} \label{thm:integrability1}  
Let $\omega_j$ be a conjugate direction determined by~{\rm(\ref{firstthree})}.
Then provided $X<0$, the tensor field $\omega_{ij}$ is symmetric in its
indices if and only if
\begin{eqnarray}
f^{ijk}f_if_jf_k + f^{ijk}f_i\omega_j\omega_k + 
2f^{ij}f_j{}^kf_if_k - 2 f^{ij}f_j{}^k\omega_i\omega_k & = & 0 \label{eq4} \\
f^{ijk}f_if_j\omega_k + f^{ijk}\omega_i\omega_j\omega_k + 
4 f^{ij}f_j{}^kf_i\omega_k & = & 0. \label{eq5}
\end{eqnarray}
\end{thm}
\begin{proof}  Since $X\neq 0$, the identity of Lemma~\ref{magic}, namely
\begin{equation} \label{identity}
f^jf_j X + 12 T_{ijk}T^{ijk} = 0\,,
\end{equation}
where $T_{ijk} = f_{[i}\omega_jf_{k]l}\omega^l$, shows that the vector field
$f^{ij}\omega_j$ is independent of $f^i$ and $\omega^i$. Therefore, the tensor
field $\omega_{ij}$ is symmetric in its indices if and only if
\begin{equation} \label{symmetric}
u^iv^j(\omega_{ij} - \omega_{ji}) = 0\,,
\end{equation}
where $u^i$ and $v^j$ are any vector fields taken from the set
$\{f^i,\omega^i,f^{ij}\omega_j\}$. Looking back at (\ref{diff1}), which was
obtained by differentiating (\ref{omega}), we see that
$$f^i\omega^j(\omega_{ij}-\omega_{ji})=f^{ij}f_if_j+f^{ij}\omega_i\omega_j.$$
This already vanishes by assumption. It is our third equation
from~(\ref{firstthree}). Differentiating this third equation gives 
\begin{eqnarray*}
 0 & = & f^i\nabla_i(f^{jk}f_jf_k + f^{jk}\omega_j\omega_k) \\
 & = & f^{ijk}f_if_jf_k + f^{ijk}f_i\omega_j\omega_k + 
 2f^{jk}f^i{}_jf_if_k + 2 f^{jk}\omega_{ij}f^i\omega_k\,.
\end{eqnarray*}
We notice that the last term $f^{jk}\omega_{ij}f^i\omega_k$ occurs as the first
component of the symmetry condition 
$f^if^{jk}\omega_k(\omega_{ij} - \omega_{ji}) = 0$, which therefore holds
if and only if
$$
f^{ijk}f_if_jf_k + f^{ijk}f_i\omega_j\omega_k + 2f^{jk}f^i{}_jf_if_k + 
2 f^{jk}\omega_{ji}f^i\omega_k = 0\,,
$$
where we have replaced $\omega_{ij}$ by $\omega_{ji}$ in the last term. But now
(\ref{diff1}) shows that we can replace $\omega_{ji}f^i$ 
with~$-f_{ji}\omega^i$. This yields~(\ref{eq4}). Similarly, the equation 
$$0 = \omega^i\nabla_i(f^{jk}f_jf_k +f^{jk}\omega_j\omega_k)=\cdots$$ 
shows that the final symmetry condition 
$\omega^if^{jk}\omega_k(\omega_{ij}-\omega_{ji})=0$ reduces to~(\ref{eq5}).
\end{proof} 
\begin{cor}\label{allfiveequations}
Locally, a smooth function $f$ with $X<0$ admits a smooth conjugate if and only
if there is a smooth solution $\omega_i$ of the equations
{\rm(\ref{firstthree}), (\ref{eq4})} and\/~{\rm(\ref{eq5})}.
\end{cor}
\begin{proof} Symmetry of $\omega_{ij}$ is precisely the condition that
$\omega_i$ be exact and, therefore, locally of the form $\nabla_ig$ for some
smooth function~$g$.
\end{proof}
Of course, we know that equations (\ref{firstthree}) admit smooth solutions
when $X<0$ so the only issue is whether we can find a solution for which
(\ref{eq4}) and (\ref{eq5}) are also satisfied. Also, if $\omega_i$ is a
solution then so is~$-\omega_i$.

\section{Resolution of the equations: the generic case} \label{sec:resolution}
Throughout this section we shall suppose that $X<0$. Recall that 
under this hypothesis $f$ has four conjugate directions at each point, 
occurring in two pairs that differ only by sign. In other words, the solutions 
of the equations (\ref{firstthree}) have the form $\{\pm\omega_i,\pm\eta_i\}$ 
for $\omega_i$ and $\eta_i$ smooth linearly independent $1$-forms. Let us 
consider the expressions
$$\begin{array}{rcl}
p^+&\equiv&f^{ijk}f_if_jf_k + f^{ijk}f_i\omega_j\omega_k + 
2f^{ij}f_j{}^kf_if_k - 2 f^{ij}f_j{}^k\omega_i\omega_k\\
p^-&\equiv&f^{ijk}f_if_jf_k + f^{ijk}f_i\eta_j\eta_k + 
2f^{ij}f_j{}^kf_if_k - 2 f^{ij}f_j{}^k\eta_i\eta_k\\
q^+&\equiv&f^{ijk}f_if_j\omega_k + f^{ijk}\omega_i\omega_j\omega_k + 
4 f^{ij}f_j{}^kf_i\omega_k\\
q^-&\equiv&f^{ijk}f_if_j\eta_k + f^{ijk}\eta_i\eta_j\eta_k + 
4 f^{ij}f_j{}^kf_i\eta_k
\end{array}$$
According to Corollary~\ref{allfiveequations} and the discussion that 
immediately follows it, we now know that $f$ admits a 
conjugate if and only if 
$$p^+=q^+=0\qquad\mbox{or}\qquad p^-=q^-=0.$$
These two possibilities are captured by the following theorem.
\begin{thm}\label{mainthm}
Locally, a smooth function $f$ with $X<0$ admits a smooth conjugate if and 
only if 
$$p^+p^-=0\qquad q^+q^-=0\qquad (p^+q^-)^2+(p^-q^+)^2=0.$$
\end{thm}
\begin{proof}Evidently, the vanishing of these three quantities is equivalent 
to $p^+=q^+=0$ or $p^-=q^-=0$. 
\end{proof}
The condition $p^+p^-=0$ was already resolved in~\cite{be}. We recapitulate and
refine the argument as follows. Firstly, we write $p^+$ using normal
coordinates (\ref{nor}) to discover that
\begin{equation}\label{pplus}p^+=p_e+p_o\omega_1\omega_2\end{equation}
where
\begin{equation}\label{pe}\begin{array}{l}
p_e=f_3{}^3f_{333}+2f_3{}^2(f_{13}{}^2+f_{23}{}^2+f_{33}{}^2)\\
\phantom{p_e}
+(f_3f_{113}-2f_{11}{}^2-2f_{13}{}^2)\omega_1{}^2
+(f_3f_{223}-2f_{22}{}^2-2f_{23}{}^2)\omega_2{}^2
\end{array}\end{equation}
and 
$$p_o=2f_3f_{123}-4f_{13}f_{23}.$$
In normal co\"ordinates $(\eta_1,\eta_2)=(\pm\omega_1,\mp\omega_2)$. It 
follows that 
\begin{equation}\label{pminus}p^-=p_e-p_o\omega_1\omega_2\end{equation}
and hence that
\begin{equation}\label{ppluspminus}
p^+p^-=p_e{}^2-p_o{}^2\omega_1{}^2\omega_2{}^2.
\end{equation} 
But, since $\omega_i$ is subject to (\ref{firstthree}), we know that
$\omega_1{}^2$ and $\omega_2{}^2$ are determined in normal co\"ordinates
by~(\ref{solution}). In \cite{be} we used this to eliminate $\omega_1{}^2$ and
$\omega_2{}^2$ from $p_e$ in (\ref{pe}) and then from $p^+p^-$ in
(\ref{ppluspminus}) to discover by trial and error that $Y^2p^+p^-$ could be
written as an explicit Riemannian invariant in the derivatives of~$f$, where
$Y$ is the invariant $Z^2-2JX$ from~(\ref{thisisY}). We can argue more
systematically as follows. Firstly, we may obtain $\eta_i$ from $\omega_i$ 
without recourse to normal co\"ordinates.
\begin{lemma}\label{formulaforeta}
The conjugate direction $\eta_i$ is determined by the conjugate direction 
$\omega_i$ via the formula
\begin{equation}\label{eta}\sqrt{Y}\eta_i=
2f^{jk}f_j\omega_kf_i+(Z-2f^{jk}f_jf_k)\omega_i-2Jf_i{}^j\omega_j.
\end{equation}\end{lemma}
\begin{proof} Since it is evidently co\"ordinate-free, we may verify this
formula in normal co\"ordinates~(\ref{nor}). Substituting from (\ref{thisisZ}) 
we see that the right hand side of (\ref{eta}) becomes
$$2(f_{13}f_3\omega_1+f_{23}f_3\omega_2)f_i
+f_3{}^2(f_{11}+f_{22})\omega_i-2f_3{}^2f_i{}^j\omega_j$$
In more detail,
\begin{center}\begin{tabular}{c|c}
$i$&right hand side of (\ref{eta})\\ \hline
$1$&$f_3{}^2(f_{11}+f_{22})\omega_1-2f_3{}^2(f_{11}\omega_1)=
f_3{}^2(f_{22}-f_{11})\omega_1$\rule{0pt}{12pt}\\
$2$&$f_3{}^2(f_{11}+f_{22})\omega_2-2f_3{}^2(f_{22}\omega_2)=
f_3{}^2(f_{11}-f_{22})\omega_2$\rule{0pt}{12pt}\\
$3$&$2(f_{13}f_3\omega_1+f_{23}f_3\omega_2)f_3
-2f_3{}^2(f_{13}\omega_1+f_{23}\omega_2)=0$
\end{tabular}\end{center}
On the other hand, from (\ref{thisisY}) the left hand side of (\ref{eta})
becomes
$$\sqrt{f_3{}^4(f_{11}-f_{22})^2}\eta_i$$
and the whole of (\ref{eta}) reduces to 
$(\eta_1,\eta_2)=\pm(\omega_1,-\omega_2)$ depending on the sign chosen for the 
square root of~$Y$.
\end{proof}
Note that since $Y>0$ when $X<0$ we could always insist of taking the positive
square root of $Y$ in~(\ref{eta}) to obtain a consistent smooth choice of
conjugate direction $\eta_i$ once $\omega_i$ is chosen. In any case, now let us
consider $p_e$ in more detail. {From} (\ref{pplus}) and (\ref{pminus})
we see that 
\begin{equation}\label{peinvariant}\textstyle p_e=\frac12(p^++p^-).
\end{equation}
Note that $p^+$ does not see the sign of $\omega_i$ and $p^-$ does not see the
sign of $\eta_i$. Moreover, interchanging $\omega_i$ and $\eta_i$ interchanges
$p^+$ and $p^-$. Hence, from (\ref{peinvariant}) we see that $p_e$ depends only
on the derivatives of~$f$. In principle, we could now use (\ref{eta}) to
substitute for $\eta_i$ in $p^-$. We conclude that $Yp_e$ is a polynomial in
$f_i,f_{ij},f_{ijk}$, and $\omega_i$, which is actually independent of
$\omega_i$ when (\ref{firstthree}) holds. Equation (\ref{firstthree}) may now
be used to eliminate $\omega_i$ from $Yp_e$ leaving a polynomial in
$f_i,f_{ij},f_{ijk}$. In practice, this is quite an intricate matter, which we
consign to~\S\ref{sec:elim-om}. The result is:
$$\textstyle Yp_e=\frac12Y(p^++p^-)=\frac12(ZS-2XR+2XY),$$
where $R$ and $S$ are two further conformal invariants derived
in~\S\ref{conformalinvariants}.

Let us apply similar reasoning to some of the other quantities occurring above.
From (\ref{pplus}) and (\ref{pminus}) we see that
$$\textstyle p_0\omega_1\omega_2=\frac12(p^+-p^-).$$
As we have already observed, interchanging $\omega_i$ and $\eta_i$ interchanges
$p^+$ and~$p^-$, hence changing the sign of~$p^+-p^-$. As is readily verified
in normal co\"ordinates, another quantity with this property is
$$E\equiv\epsilon^{ijk}f_i\omega_jf_k{}^\ell\omega_\ell$$
where $\epsilon^{ijk}$ is a choice of volume form, uniquely normalised up to 
sign by $\epsilon^{ijk}\epsilon_{ijk}=6$. Specifically, if we further 
constrain our normal co\"ordinates (\ref{nor}) by requiring that 
$\epsilon^{123}=1$, then 
$$E=f_3(f_{22}-f_{11})\omega_1\omega_2.$$
As above, it follows that we may use (\ref{eta}) to eliminate $\eta_i$ from  
$$\textstyle YEp_o\omega_1\omega_2=\frac12E(Yp^+-Yp^-).$$
Moreover, this quantity is stable under interchange of $\omega_i$ and $\eta_i$.
It must be a polynomial in $f_i,f_{ij},f_{ijk}$ alone, which is given by:
$$\textstyle YEp_o\omega_1\omega_2=\frac12E(Yp^+-Yp^-)=-\frac14JXV,$$
where this calculation is once more detailed in~\S\ref{sec:elim-om} and $V$ is
one of our list of conformal invariants derived in~\S\ref{conformalinvariants}.
But from Lemma~\ref{magic}, we have the identity
\begin{equation} \label{identityEX}
\textstyle E^2=-\frac12J^2X.
\end{equation}
We conclude that
\begin{eqnarray*}
P & \equiv & 8Y^2p^+p^-=2Y^2(p^++p^-)^2-2Y^2(p^+-p^-)^2 \\
& = & 2(ZS-2XR+2XY)^2 + XV^2.
\end{eqnarray*}
The vanishing of $P$ is then our fourth conformally invariant condition (in
addition to the first three (\ref{firstthree})), obtained in~\cite{be}, for the
existence of a conjugate in the generic case $X<0$. We now proceed similarly to
obtain the two other conditions to provide a necessary and sufficient set of
conditions.

First we observe that $Q\equiv Y\sqrt{Y}q^+q^-$ is conformally invariant, where
we use Lemma~\ref{formulaforeta} to define $\eta_i$ by a choice of square root
for~$Y$. Certainly it is a Riemannian invariant and we shall compute it in
normal co\"ordinates~(\ref{nor}). According to the proof of
Lemma~\ref{formulaforeta}, we may take $$\sqrt{Y}=f_3{}^2(f_{22}-f_{11})\qquad
\eta_1=\omega_1\qquad\eta_2=-\omega_2,$$
in which case
$$q^+=q_1\omega_1+q_2\omega_2\quad\mbox{and}\quad q^-=q_1\omega_1-q_2\omega_2,
$$
where
$$\begin{array}{rcl}
q_1&=&f_3{}^2f_{133}+f_{111}\omega_1{}^2+3f_{122}\omega_2{}^2
+4f_3f_{13}(f_{11}+f_{33})\\
q_2&=&f_3{}^2f_{233}+f_{222}\omega_2{}^2+3f_{112}\omega_1{}^2
+4f_3f_{23}(f_{22}+f_{33})\\
\end{array}$$
so that 
$$Q=Y\sqrt{Y}q^+q^-=f_3{}^6(f_{22}-f_{11})^3
(q_1{}^2\omega_1{}^2-q_2{}^2\omega_2{}^2)$$
from which $\omega_1{}^2$ and $\omega_2{}^2$ may be eliminated 
with~(\ref{solution}). The result is a polynomial expression in $f$ and its 
derivatives. In terms of the various conformal invariants developed 
in~\S\ref{conformalinvariants} it turns out that
$$\begin{array}{rcl}
Q&=&\frac{1}{6}JZB-\frac{1}{4}JU-\frac{1}{4}ZS^2\\
&&\enskip{}+X(XZ^3-JX^2Z+6W+\frac{1}{4}JM-\frac{2}{7}ZXR+\frac{5}{7}RS\\
&&\qquad\quad{}-\frac{15}{7}N+\frac{2}{9}ZA-\frac{9}{10}F
               -\frac{2}{21}ZK+\frac{10}{21}T+\frac{6}{25}G-\frac{17}{42}JD),
\end{array}$$
as may be verified in normal form~(\ref{nor}).

The final condition $(p^+q^-)^2 + (p^-q^+)^2 = 0$ can similarly be expressed in
terms of conformal invariants; although we do not attempt to write down the
expression, we discuss how this can be done in~\S\ref{sec:elim-om}.

\section{Special cases} \label{sec:sc}

\subsection{Functions with a unique conjugate direction} \label{sec:X=0}
 Suppose now that $f$ is a function that admits a unique
conjugate direction up to sign. By (\ref{solutionsintherealcase}), this occurs
when $X=0$ and $Y>0$. We first prove an analogue of Theorem
\ref{thm:integrability1}.
\begin{thm} \label{thm:integrability2}  
Let $\omega_j$ be a conjugate direction determined by~{\rm(\ref{firstthree})},
with $X=0$ and $Y>0$. Then the tensor field $\omega_{ij}$ is symmetric in its
indices if and only if
\begin{eqnarray}
\epsilon^{ijk}f_i\omega_j
\left( Jf_k{}^{lm}f_l\omega_m - 2f_{kl}f^l(f^{mn}f_m\omega_n)\right) & = 
& 0 \label{4bis} \\
\epsilon^{ijk}f_i\omega_j
\left( Jf_k{}^{lm}\omega_l\omega_m + f_k{}^lf_l(f^{mn}f_mf_n + Z)\right) & = 
& 0  \label{5bis}
\end{eqnarray}
\end{thm}
\begin{proof} As in the proof of Theorem \ref{thm:integrability1},
$\omega_{ij}$ is symmetric in its indices if and only if
$u^iv^j(\omega_{ij}-\omega_{ji})=0$, where $u^i$ and $v^j$ are linearly
independent vector fields. However, since $X=0$, by Lemma \ref{magic}, the
vector field $f^{ij}\omega_j$ is a linear combination of $f^i$ and $\omega^i$
and we have to use an alternative. A judicious choice turns out to be the
vector field
$$
\nu^i = \epsilon^{ijk}f_jf_k{}^l\omega_l - \epsilon^{jkl}f_j\omega_kf_l{}^i\,.
$$
A short calculation using the identity
$$
\epsilon^{ijk}\epsilon_{lmn} = 6 \delta_{[l}^i\delta_m^j\delta_{n]}^k\,,
$$
shows that $\epsilon_{ijk}\nu^if^j\omega^k = - Z/J$, which is non-zero by
hypothesis (since $Y = Z^2$). In particular, $\nu^i$ has a non-zero component
orthogonal to $f^i$ and $\omega^i$. In order to bring this vector field into
play, rather than differentiate the third equation from (\ref{firstthree}), we
differentiate the equation:
\begin{equation} \label{det}
\epsilon^{ijk}f_i\omega_jf_k{}^l\omega_l = 0\,.
\end{equation}
This gives
\begin{equation} \label{diff-det}
\epsilon^{ijk}f_i\omega_jf_{klm}\omega^l +
\epsilon^{ijk}f_{mi}\omega_jf_k{}^l\omega_l - \omega_{mi}\nu^i = 0\,.
\end{equation} 
First transvect this with $f^m$. Then the resulting symmetry condition
$f^m\nu^i(\omega_{mi} - \omega_{im})=0$ holds if and only if
$$
\epsilon^{ijk}f_i\omega_jf_{klm}f^l\omega^m + 
\epsilon^{ijk}f_{im}f^m\omega_jf_k{}^l\omega_l - f^m\omega_{im}\nu^i = 0\,.
$$
But from (\ref{diff1}), the last term can be replaced by $\omega^mf_{im}\nu^i$
which is equal to $[\epsilon^{lmn}f_lf_m{}^rf_r\omega_n/(f^sf_s)]
f^{ij}f_i\omega_j$ (since $\epsilon^{ijk}f_j\omega_kf_{im}\omega^m = 0$ by
(\ref{det})). On multiplying through by $J$, we obtain the equation
\begin{equation} \label{pre-4bis}
J\epsilon^{ijk}f_i\omega_jf_k{}^{lm}f_l\omega_n - 
J\epsilon^{ijk}f_i{}^lf_lf_j{}^m\omega_m\omega_k -
(\epsilon^{ijk}f_i\omega_jf_k{}^lf_l)f^{mn}f_m\omega_n = 0\,.
\end{equation}
However, from (\ref{det}) we deduce the identity
$$
Jf_{jm}\omega^m + (f^{kl}f_kf_l)\omega_j - (f^{kl}f_k\omega_l)f_j = 0\,.
$$
Indeed, the left-hand side is both orthogonal and colinear to the span of $f_j$
and $\omega_j$. On replacing $Jf_{jm}\omega_m$ by $(f^{kl}f_k\omega_l)f_j -
(f^{kl}f_kf_l)\omega_j$ in the middle term of (\ref{pre-4bis}), we obtain
(\ref{4bis}). Similarly, on transvecting (\ref{diff-det}) with $\omega^m$, we
conclude that the symmetry condition $\omega^m\nu^i(\omega_{mi} - \omega_{im})
= 0$ is equivalent to~(\ref{5bis}).
\end{proof}

As for the generic case, we can summarise the conditions that $f$ admits a
conjugate as follows.
\begin{cor}\label{allfiveequations-bis}
Locally, a smooth function $f$ with $X=0$ admits a smooth conjugate
if and only if there is a smooth solution $\omega_i$ of the equations
{\rm(\ref{firstthree}), (\ref{4bis})} and\/~{\rm(\ref{5bis})}.
\end{cor}

We can express these conditions in terms of the derivatives of $f$
either by using invariant arguments, or by expressing them in normal
co\"ordinates. To do this invariantly, the following lemma can be employed to
eliminate quadratic terms in $\omega^i$. 
\begin{lemma}
\label{lem:identity-quad} Suppose $X= 0$ and $Y\neq 0$. 
Let $Q^{ij}$ be any symmetric form. Then
\begin{equation} \label{identity-quad}
ZQ^{ij}\omega_i\omega_j = -ZQ^{ij}f_if_j + 2JQ^{ij}f_if_j{}^kf_k +
J^2(f_k{}^kQ_l{}^l - Q^{kl}f_{kl})\,.
\end{equation}
\end{lemma}
\begin{proof}  
Recall that $E\equiv\epsilon^{ijk}f_i\omega_jf_k{}^\ell\omega_\ell$ satisfies
$E^2 = - J^2X/2$, so that
\begin{equation} \label{identity-E}
X=0 \quad \Leftrightarrow \quad E=0 \quad \Leftrightarrow \quad 
Jf_{jk}\omega^k + (f^{kl}f_kf_l)\omega_j - (f^{kl}f_k\omega_{\ell})f_j=0\,,
\end{equation}
where the latter equality occurs since the LHS is both orthogonal and colinear
to the span of $f_j$ and $\omega_j$. We then apply this to the identity given
by transvecting $f_{[i}\omega_jf_k{}^kQ_{l]}{}^l = 0$ with $f^i\omega^j$. An
alternative proof is simply to check that the formula holds in the Riemannian
normalisation.
\end{proof}
 
Equation (\ref{4bis}) can now be written in the form
$Q^{ij}\omega_i\omega_j=0$, where
\begin{eqnarray*}
Q^{ij} &  = & -\epsilon^{ikl}f_k(Jf_l{}^{mj}f_m -
2f_{lm}f^m(f^{nj}f_n)) \\
 & & \qquad -\epsilon^{jkl}f_k(Jf_l{}^{mi}f_m -
2f_{lm}f^m(f^{ni}f_n)),
\end{eqnarray*}
which, by Lemma \ref{lem:identity-quad} can be written
as an invariant expression in the derivatives of $f$. However, it is more
direct and somewhat simpler to just write out (\ref{4bis}) in the Riemannian
normalisation.
 
From the proof of Lemma \ref{magic}, we see that $X=0$ implies that the product
$\omega_1\omega_2 = 0$. Thus (\ref{4bis}) becomes:
$$
f_3{}^3(\omega_1{}^2-\omega_2{}^2)(f_3f_{123} -2f_{13}f_{23}) = 0\,.
$$
But since $Y\neq 0$ ($f_{22}-f_{11} \neq 0$), $\omega_1{}^2-\omega_2{}^2 = J$
and this is equivalent to
$$
f_3{}^5(f_3f_{123} - 2f_{13}f_{23}) = 0\,,
$$
which we recognize to be a multiple of $V$ (which is given in normal 
co\"ordinates by 
$4J^2f_3(f_{22}- f_{11})(f_3f_{123}-2f_{13}f_{23})$). Thus (\ref{firstthree})
and (\ref{4bis}) correspond to the conformally invariant condition $V=0$.

We give an invariant treatment of (\ref{5bis}) as follows. Differentiate the
right-hand identity of (\ref{identity-E}):
\begin{eqnarray}
0 & \!\!=\!\! & \nabla_i (Jf_{jk}\omega^k + (f^{kl}f_kf_l)\omega_j 
                              - (f^{kl}f_k\omega_l)f_j) \nonumber \\
  & \!\!=\!\! & 2(f_{il}f^l)f_{jk}\omega^k + Jf_{ijk}\omega^k + 
  J f_j{}^k\omega_{ik} + f_{ikl}f^kf^l\omega_j + 2 f^{kl}f_{ik}f_l\omega_j 
  \nonumber \\
  & & \enskip{}+f^{kl}f_kf_l\omega_{ij} - f_{ikl}f^k\omega^lf_j 
  - f^{kl}f_{ik}\omega_l f_j - f^{kl}f_k\omega_{il}f_j 
  - f^{kl}f_k\omega_l f_{ij}\,.
\label{identity-E2}  
\end{eqnarray}
Note that for the moment we do not assume symmetry of $\omega_{ij}$.  

Recall the fundamental identities: $\omega^{ij}\omega_j = f^{ij}f_j$ and
$\omega^{ij}f_j = - f^{ij}\omega_j$. Transvect (\ref{identity-E2}) with
$\omega^j$ to obtain:
\begin{eqnarray*}
0 & = & -2f_{il}f^lf^{jk}f_jf_k + J f_{ijk}\omega^j\omega^k 
                            + J(f_j{}^k\omega^j)\omega_{ik} \\
  & & \qquad  + J f_{ikl}f^kf^l + 2Jf^{kl}f_{ik}f_l 
              + (f^{kl}f_kf_l)(f_{ij}f^j) - f^{kl}f_k\omega_lf_{ij}\omega^j\,.
\end{eqnarray*}
{From} (\ref{identity-E}), 
$Jf_j{}^k\omega^j = (f^{lm}f_l\omega_m)f^k - (f^{lm}f_lf_m)\omega^k$, so that
\begin{eqnarray*}
 J(f_j{}^k\omega^j)\omega_{ik} 
 & = & (f^{lm}f_l\omega_m)\omega_{ik}f^k - (f^{lm}f_lf_m)\omega_{ik}\omega^k \\
 & = & - (f^{lm}f_l\omega_m)f_{ik}\omega^k - (f^{lm}f_lf_m)f_{ik}f^k\,,
\end{eqnarray*}
which gives the identity:
$$
Jf_{ijk}\omega^j\omega^k + J f_{ijk}f^jf^k - 2(f^{kl}f_kf_l)f_{ij}f^j 
         - 2(f^{kl}f_k\omega_l)f_{ij}\omega^j + 2Jf^{kl}f_{ik}f_l\! =\! 0.
$$
{From} this, we deduce that (\ref{5bis}) has the equivalent form:
\begin{eqnarray}
 & & \hspace{-20pt}\epsilon^{ijk}f_i\omega_j
 \left( -Jf_k{}^{lm}f_lf_m - 2Jf_{kl}f^{lm}f_m + f_k{}^lf_l(3f^{mn}f_mf_n + Z)
 \right) = 0 \Leftrightarrow \nonumber \\
 & & \qquad  \textstyle\epsilon^{ijk}f_i\omega_j
 \left(-\sigma_k+J(J\nabla_k(\Delta f) - \frac{1}{2} \Delta f\nabla_kJ)\right)
 = 0\,, \label{5bisbis}
\end{eqnarray} 
where $\sigma_k$ is the conformally invariant $1$-form given by Theorem
\ref{usefullist} of Appendix \ref{conformalinvariants}. Even though
$J\nabla_k(\Delta f) - \frac{1}{2} \Delta f\nabla_kJ$ is not itself conformally
invariant, its component orthogonal to the span of $f_i$ and $\omega_i$ is, so
the left-hand side of (\ref{5bisbis}) is conformally invariant. Now square this
and use Lemma \ref{lem:identity-quad} to eliminate quadratic terms in
$\omega_i$. We obtain an identity involving only the derivatives of $f$, which
we identify in terms of conformal invariants as:
\begin{equation} \label{X=0-5th}\textstyle
\frac{25}{14}N+\frac{3}{5}G+\frac{3}{4}F + \frac{1}{21}T - \frac{17}{21}ZK 
- \frac{7}{9}ZA = 0\,.
\end{equation} 
\begin{cor} \label{cor:unique} Locally, a smooth function $f$ with $X=0$ admits a smooth conjugate
if and only if\, $V\equiv 0$ and {\rm (\ref{X=0-5th})} are satisfied.
\end{cor}

\subsection{Functions that admit infinitely many conjugates}
When $X$ and $Y$ both vanish, the function $f$ admits infinitely many
conjugate directions. The following gives a complete description. 
\begin{thm}
\label{thm:XYzero} Suppose $f$ is a smooth real-valued non-constant function
such that its invariants $X$ and $Y$ both vanish. Then, up to scale and
conformal transformation, $f$ is one of the following
\begin{equation}\label{XYzero}
x_1 \quad \log(x_1{}^2+x_2{}^2+x_3{}^2)\quad
\arctan\left(\frac{x_3}{x_2}\right)
\quad\frac{x_1}{x_1{}^2+x_2{}^2+x_3{}^2}.\end{equation}
\end{thm}  
\begin{proof}
{From} (\ref{aha}) we deduce immediately that $\phi_{ij}=0$. But 
$$\phi_{ij}=\mbox{the symmetric trace-free part of }J^2\nabla_i[J^{-1}f_j]$$
whose vanishing is precisely saying that $J^{-1}f_j$ is a conformal Killing
field $V_j$ all of which can be written down explicitly.
Following~\cite{laplace},
$$\textstyle V_j=-s_j-m_{jk}x^k+\lambda x_j+x_jr_kx^k-\frac12r_jx_kx^k$$
where $s_j$ and $r_j$ are arbitrary vectors, $\lambda$~is a arbitrary constant, 
and $m_{ij}$ is an arbitrary skew matrix. We may invert
$$V_j=(f^kf_k)^{-1}f_j\iff f_j=(V^kV_k)^{-1}V_j$$
and inquire whether $f_j$ is closed. As a condition on~$V_j$, this reads
\begin{equation}\label{testingV}
V^kV_k\nabla_{[i}V_{j]}+2V^kV_{[i}\nabla_{j]}V_k=0,\end{equation}
the consequences of which are best viewed using a normal form for~$V_j$ such as
those provided by Theorem~\ref{fivebyfivecase} 
in~\S\ref{normalised_conf_Killing}. Specifically, matrices of the
form (\ref{firsttype}) provide conformal Killing fields of the form 
$$\lambda\Big(x_1\frac{\partial}{\partial x_1}
+x_2\frac{\partial}{\partial x_2}+x_3\frac{\partial}{\partial x_3}\Big)
+\mu\Big(x_2\frac{\partial}{\partial x_3}
-x_3\frac{\partial}{\partial x_2}\Big)$$
in accordance with the conventions of~\cite{laplace}. However, only when
$\mu=0$ or $\lambda=0$ is (\ref{testingV}) satisfied. When both vanish, we
obtain the linear functions which are equivalent under scaling and conformal
transformation to the first of~(\ref{XYzero}). Otherwise we obtain the second
two, respectively. Matrices from the next group provide nothing new but
matrices of the form (\ref{thirdtype}) correspond to the conformal Killing
fields $$\mu\Big(x_2\frac{\partial}{\partial x_3} -x_3\frac{\partial}{\partial
x_2}\Big) -(x_1{}^2-x_2{}^2-x_3{}^2)\frac{\partial}{\partial x_1}
-2x_1x_2\frac{\partial}{\partial x_2}-2x_1x_3\frac{\partial}{\partial x_3}$$
and (\ref{testingV}) is satisfied precisely when $\mu=0$. This gives rise to
the final possibility for $f$ in the list~(\ref{XYzero}).
\end{proof}

In fact, all of the functions with $X=Y=0$ admit, not only infinitely many
conjugate directions, but infinitely many conjugates. According to
Theorem~\ref{thm:XYzero}, it suffices to check this for the four 
cases~(\ref{XYzero}). The first three of these are discussed in detail elsewhere in
this article, specifically in \S\ref{lin_and_quad}, \S\ref{subsec:spherical},
and \S\ref{cylindrical} respectively. Finally, the functions 
$$f=\frac{x_1}{x_1{}^2+x_2{}^2+x_3{}^2}\qquad
g=\frac{x_2\cos\theta+x_3\sin\theta}{x_1{}^2+x_2{}^2+x_3{}^2}$$
form a conjugate pair for any~$\theta$.

\subsection{Functions of two variables that admit a conjugate in
${\mathbb{R}}^3$} 
Let $f = f(x_2, x_3)$ be a function of two variables only.
Then many conformal invariants simplify and in the case of a unique conjugate
direction, the equations have a simple interpretation. As a first observation,
it is easily checked that $X$ factors as a product:
$$
X = (\Delta f)( f^i\nabla_iJ)\,,
$$
so that we also have
$$\textstyle
Z = \frac12 f^i\nabla_iJ + J\Delta f \,,\quad 
Y = \left(\frac12 f^i\nabla_iJ - J\Delta f\right)^2\,.
$$
Furthermore, by its expression in the Riemannian normalisation, one sees that
$V\equiv 0$. In particular, the fourth condition for a conjugate: $P\equiv 0$
simplifies to
$$
ZS-2XR+2XY=0\,.
$$ 

Now suppose $X=0$ and $Y>0$. Then either $\Delta f = 0$, in which case 
$\omega = (0, -f_3, f_2)$ is, up to sign, the unique integrable conjugate
direction and we are in the case of a planar function with planar conjugate, or
$f^i\nabla_i J=0$ and $\Delta f \neq 0$. We can now exploit
Theorem~\ref{thm:integrability2}. Since (\ref{4bis}) is equivalent to 
$V\equiv 0$, this is vacuous. However, (\ref{5bis}) now comes into play. By
going into the Riemannian normalisation, one sees that the third order terms of
this equation vanish, and it becomes:
$$
(\epsilon^{ijk}f_i\omega_jf_k{}^lf_l)(f^{mn}f_mf_n+Z)=0\,.
$$
However, since $\Delta f \neq 0$, it is also the case that 
$f^{mn}f_mf_n+Z\neq 0$ and the equation becomes
$$
\epsilon^{ijk}f_i\omega_jf_k{}^lf_l=0\,.
$$
Let us write this out explicitly in co\"ordinates:
$$
-\omega_1f_2f_3{}^lf_l + \omega_1f_3f_2{}^lf_l = 0\,.
$$
But $\omega_1$ must be non-zero otherwise we are once more in the situation of
a planar function with a planar conjugate whence $\Delta f=0$,
contrary to our hypothesis. On combining this with the condition 
$f^i\nabla_iJ = 0$, we obtain the simultaneous equations in $f_2{}^kf_k$ and
$f_3{}^kf_k$:
$$
\left\{ \begin{array}{rcl}
f_3f_2{}^kf_k - f_2f_3{}^kf_k & = & 0 \\
f_2f_2{}^kf_k + f_3f_3{}^kf_k &  = & 0 
\end{array} \right.
$$
Since $f_2{}^2+f_3{}^2\neq 0$, these only admit the solution 
$f_2{}^kf_k = f_3{}^kf_k = 0$. But this implies that
$$
\nabla_l(f^kf_k) = 0 \quad\Leftrightarrow\quad ||\nabla f|| = {\rm constant}\,.
$$ 
The unique conjugate direction is thus given up to sign by 
$$\omega = (\sqrt{f_2{}^2+f_3{}^2},0,0).$$
Furthermore this case occurs precisely when $f$ satisfies the eikonal equation
$||\nabla f||^2 =$ constant. This should be compared with the example of a
function having spherical symmetry as discussed in \S\ref{subsec:spherical}
below, where now the conjugate must satisfy an eikonal equation, even though
there is no conformal transformation which sends concentric spheres to parallel
planes.

\section{Some examples} \label{sec:examples}
In general, it is not the case that a function will admit a conjugate, even
locally. For example, the function $f = x_1x_2x_3$ has the property that 
$X = 6f^2$. In particular $X$ cannot be $\leq 0$ on any open set, so that $f$
does not admit a conjugate on any open set.

Recall from the Introduction that the pair $(f,g)$ of a function and its
conjugate define a semi-conformal mapping into ${\mathbb{R}}^2$. In the
analytic category, such mappings arise (i) as the extension to the boundary at
infinity of a harmonic morphism on the associated heaven space of the domain,
see \cite{Ba-Pa}; (ii) from local CR hypersurfaces in the standard
Levi-indefinite hyperquadric in ${\mathbb{C}}P_3$, see \cite{Ba-Ea}. The latter
perspective leads to an explicit construction of semiconformal mappings from a
holomorphic function of two complex variables, which, in a first form was given
in \cite{Nu} then refined in \cite{Ba-Ea}. In what follows, we highlight some
particular cases of interest when a function $f$ admits a conjugate function.

\subsection{Linear and quadratic functions}\label{lin_and_quad}
Any linear function $f$ admits infinitely many conjugate functions, also
linear; indeed the two invariants $X$ and $Y$ both vanish identically. The only
quadratic function that admits a conjugate is, up to isometries and scaling,
$f = x_1{}^2 - x_2{}^2 - x_3{}^2$. Note that $f$ has an isolated critical point
at the origin, however its conjugate $g = x_1\sqrt{x_2{}^2 + x_3{}^2}$,
although of class $C^1$ at the origin, is not smooth there. It is unknown if a
pair of smooth conjugate functions $(f,g)$ can have an isolated critical point.
When they are harmonic and so determine a harmonic morphism, this is
impossible~\cite{thebook}.

\subsection{Cylindrical symmetry}\label{cylindrical}
Let $r^2=x_2{}^2+x_3{}^2$ and suppose that $f=f(r)$ so that its level sets are
concentric cylinders. Then by solving the equations (\ref{firstthree}), we
obtain the conjugate direction:--
$$
(\omega_1, \omega_2, \omega_3) = 
\left( \sqrt{f^{\prime \, 2} + rf^{\prime}f^{\prime\prime}} , 
x_3 \sqrt{\frac{-f^{\prime}f^{\prime\prime}}{r}}, 
- x_2 \sqrt{\frac{-f^{\prime}f^{\prime\prime}}{r}}\right)
$$
whose four-valuedness corresponds to taking different signs for the square
roots. Then for any branch, $d\omega = 0$ if and only if
$$
f^{\prime \, 2} + 2f^{\prime}f^{\prime\prime} = C,
$$
where $C$ is a constant which is $\geq 0$.   This has as first integral:--
\begin{equation} \label{first-int}
f^{\prime \, 2} = \frac{A}{r^2} + C,
\end{equation}
where $A\geq 0$ is a constant, and $\omega$ is now given by
$$
(\omega_1, \omega_2, \omega_3) = 
\left( \sqrt{C} , \frac{x_3\sqrt{A}}{r^2}, - \frac{x_2\sqrt{A}}{r^2}\right)
$$
Then $X = 2Cf^{\prime}f^{\prime\prime}/r = - 2AC/r^4$ is $\leq 0$ with the
inequality strict provided neither of $A$ nor $C$ vanish.

In fact we can integrate (\ref{first-int}) explicitly to obtain
$$
f = \left\{ \begin{array}{l} \sqrt{A} \ln\left\{ \frac{\sqrt{A + Cr^2} -
\sqrt{A}}{\sqrt{C}r}\right\} + \sqrt{A + Cr^2} \quad (C>0) \\
\sqrt{A}\ln r \quad (C = 0)
\end{array} \right.
$$
The conjugate function is given by 
$g = \sqrt{C} x_1 - \sqrt{A} \arctan (x_3/x_2)$, interpolating between the
two special case given by $A = 0$ $(f = \sqrt{C} r)$ and $C = 0$ $(f = \sqrt{A}
\ln r)$. In fact the mapping $(f,g)$ has fibres which are helices lying on the
cylinders $r =$ constant. When $C=0$ these helices become circles lying in
planes orthogonal to the $x_1$-axis and when $A=0$ they become lines parallel
to the $x_1$-axis. Geometrically, we can interpret the four-valuedness of
$\omega$ as corresponding to the choice of a right-hand screw or a left-hand
screw for the helices, together with a choice of orientation. In the special
cases we obtain just two equal and opposite directions.

\subsection{Spherical symmetry} \label{subsec:spherical} Let
$r^2=x_1{}^2+x_2{}^2+x_3{}^2$ and suppose that $f = f(r)$ depends on the radial
coordinate only. Then
$$
X = 
2f^{\prime}(r)^2 \left( f^{\prime\prime}(r) + \frac{f^{\prime}(r)}{r}\right)^2
$$
so that if $f$ is to admit a conjugate, the necessary condition $X\leq 0$
forces $f$ to be either constant or to satisfy the differential equation
$$
f^{\prime\prime}(r) + \frac{f^{\prime}(r)}{r} = 0\,.
$$
This has general solution $f = A\log r + B$, where $A$ and $B$ are arbitrary
constants. For convenience, we take $f = \log r$. Note that spherical symmetry
implies that $Y \equiv 0$ and so there are infinitely many conjugate
directions. In fact any conjugate function $g$ must satisfy 
$\partial g / \partial r = 0$ and $||\nabla g|| = 1/r$. Thus $g$ is determined
by its values on say the sphere $r = 1$, where it must satisfy the equation
$||\nabla g||=1$. Such an equation is know as an \emph{eikonal equation} and
solutions are determined by initial data on a hypersurface (i.e.~a curve) in
the sphere $S^2$. It should be noted that the sphere $S^2$ does not admit a
nowhere vanishing vector field and since we require $||\nabla g||= 1$, then $g$
cannot be globally defined on $S^2$. Thus even though the function $f$ defined
on ${{\mathbb R}^3}\setminus \{ 0\}$ admits infinitely many different conjugate
functions in a neighbourhood of any point of its domain, the domain of any of
these conjugate functions cannot coincide with that of~$f$.

\subsection{An Ansatz}
The following Ansatz provides a method of obtaining many pairs of conjugate
functions. Let $h(x,y)$ satisfy the partial differential equation:
\begin{equation} \label{ansatz}
\left(\frac{\partial h}{\partial x}\right)^2 + 
4y\left(\frac{\partial h}{\partial y}\right)^2 + 
4h\frac{\partial h}{\partial y} = 0\,.
\end{equation}
Then the functions 
$$
\left\{ \begin{array}{rcl} f & = & x_2h(x_1, {x_2}^2 + {x_3}^2) \\
g & = & x_3h(x_1, {x_2}^2 + {x_3}^2)\end{array} \right.
$$
are conjugate. For example, by taking $h = (x^2/y) + 1$, we obtain the pair of
conjugate functions of the Introduction. A straightforward calculation shows
that the only product solutions $h(x,y) = u(x)v(y)$ to (\ref{ansatz}), have the
form
$$
h = \frac{be^{cx}e^{\sqrt{1 - c^2y}}}{1 + \sqrt{1 - c^2y}},
$$
where $b$ and $c$ are constants. In fact, with reference to ~\S\ref{sec:X=0},
every solution obtained by this Ansatz satisfies $X\equiv 0$.

\section{Invariants of the conjugate}  \label{sec:invariantsconjugate}
For a function $f$ which admits a conjugate $g$, we can ask which of its
properties are shared by its conjugate. More specifically, can we express the
conformal invariants of $g$ in terms of those of $f$\,? For the invariant $X$,
this turns out to be simply done. In order to be clear on which invariants are
being considered, in this section we shall write $X(f)$ and $X(g)$ and so on,
for the invariants of the respective functions.

\begin{thm}\label{thm:Xconjugates}  
If $f$ admits a conjugate function $g$, then $X(f) = X(g)$.
\end{thm} 
\begin{proof}  In addition to (\ref{firstthree}), we have the identities:
$$
g^{ij}f_j+f^{ij}g_j = 0 \quad {\rm and} \quad g^{ij}g_j = f^{ij}f_j\,.
$$
Set $v_i = \epsilon_{ijk}f^jg^k$.  
Then we can decompose $g_{ij}$ in terms of a symmetric basis:
\begin{eqnarray*}
g_{ij} & = & \frac{1}{J^2}(g^{kl}f_kf_l)f_if_j 
+ \frac{1}{J^2}(g^{kl}g_kg_l)g_ig_j + \frac{1}{J^4}(g^{kl}v_kv_l)v_iv_j \\
 & & + \frac{2}{J^2}(g^{kl}f_kg_l)f_{(i}g_{j)} 
+ \frac{2}{J^3}(g^{kl}f_kv_l)f_{(i}v_{j)} 
+ \frac{2}{J^3}(g^{kl}g_kv_l)g_{(i}v_{j)} \\
 & = & \frac{1}{J}(g^k{}_k)(J\delta_{ij} - f_if_j-g_ig_j) 
 - \frac{2}{J}f_{k(i}g^kf_{j)} + \frac{2}{J} f_{k(i}f^kg_{j)} \\
 & & + \frac{1}{J^2}(f^{kl}f_kg_l)(f_if_j-g_ig_j) 
 - \frac{2}{J^2}(f^{kl}f_kf_l)f_{(i}g_{j)}\,.
\end{eqnarray*}
As a first application of this formula, we deduce the identity:
\begin{equation} \label{identity-conjugates}
f^{ij}g_{ij} - (f^i{}_i)(g^j{}_j) = 0\,.
\end{equation}
Furthermore
$$
g^{ij}g_{ij} =(g^k{}_k)^2 + \frac{2}{J}f_{kj}g^kf^{lj}g_l 
+ \frac{2}{J} f_{kj}f^kf^{lj}f_l - \frac{2}{J^2}(f^{kl}f_kg_l)^2 
- \frac{2}{J^2}(f^{kl}f_kf_l)^2\,,
$$
which implies that
$$
X(g) = - 2f_{kj}g^kf^{lj}g_l + \frac{2}{J}(f^{kl}f_kg_l)^2 
+ \frac{2}{J}(f^{kl}f_kf_l)^2\,.
$$
In normal co\"ordinates, on applying (\ref{solution}), the RHS equals 
$$
-2g_1{}^2f_{11}^2 - 2g_2{}^2f_{22}{}^2 + 2f_3{}^2f_{33}{}^2 = 
2f_3{}^2(f_{11} + f_{33})(f_{22}+ f_{33})\,,
$$ which is precisely $X(f)$.
\end{proof}

\begin{cor} If $f$ admits a conjugate function $g$, then for any 
$\epsilon \in {\mathbb{R}}$, the function $f+\epsilon g$ admits $g-\epsilon f$
as a conjugate and $X(f+\epsilon g) = (1+\epsilon^2)^2X(f)$.
\end{cor}

\begin{proof}
That $f+\epsilon g$ and $g-\epsilon f$ are conjugates, is easily checked.  Then
\begin{eqnarray*}
X(f+\epsilon g) & \!\!=\!\! & X(f) \\
& & + \epsilon \{ 4(f_i{}^jg_j+4g_i{}^jf_j)f^{ik}f_k - 2J(g^{ij}f_{ij} 
                                                  - (f^i{}_i)(g^j{}_j))\} \\
& & + \epsilon^2\{ 4f_i{}^jf_jg^{ik}g_k+4f_i{}^jg_jg^{ik}f_k
                               +2g_i{}^jf_jg^{ik}f_k+2f_i{}^jg_jf^{ik}g_k \\
& & \qquad - J[f^{ij}f_{ij} + g^{ij}g_{ij} - (f^i{}_i)^2 - (g^j{}_j)^2]\} \\
& & + \epsilon^3\{ 4(f_i{}^jg_j+4g_i{}^jf_j)g^{ik}g_k 
                                 - 2J(g^{ij}f_{ij} - (f^i{}_i)(g^j{}_j))\}\\
& &  \qquad + \epsilon^4X(g)\,.
\end{eqnarray*}
But the coefficients of the odd powers of $\epsilon$ vanish on account of
(\ref{firstthree}) and (\ref{identity-conjugates}), so from Theorem
\ref{thm:Xconjugates}, we obtain
$$
X(f+\epsilon g) = X(f) + \epsilon^2(X(f)+X(g)) + \epsilon^4X(g) =
(1+\epsilon^2)^2X(f)\,.
$$
\end{proof}

Note that if we view the pair $(f,g)$ of a function and its conjugate as
defining a semiconformal map into ${\mathbb{R}}^2$, then the replacement of
$(f,g)$ by $(f+\epsilon g, g-\epsilon f)$ amounts to multiplication of $f+ig$
by $1-i\epsilon$ when we identify ${\mathbb{R}}^2$ with the complex plane
${\mathbb{C}}$. Indeed, semiconformality is preserved under conformal
transformations of both the domain and codomain.

To calculate the invariant $Z(g)$ in terms of invariants of $f$ turns out to be
more challenging. In fact $Z(g)$ depends on the choice of conjugate direction,
so that, in the generic case, the appropriate quantity to consider is the
product $\sqrt{Y}Z(\omega )Z(\eta )$. This can be calculated by the methods of
\S\ref{sec:elim-om} to produce an expression involving third order derivative
of $f$ which we don't attempt to write down. On the other hand, information
about $Z(g)$ can be obtained as in the above Corollary. 
\begin{lemma}\label{lem:Zfg} If $f$ admits a conjugate $g$, then we have
$$
Z(f+\epsilon g) = (1+\epsilon^2)(Z(f) + \epsilon Z(g))\,.
$$
Furthermore, 
$$
Z(g) = \frac{d}{d\epsilon}Z(f+\epsilon g)\vert_{\epsilon =0}\,;
$$
that is, $Z(g)={\mathcal Z}_f(g)$ where ${\mathcal Z}_f$ is the linearisation
of the operator $Z$ at $f$.
\end{lemma}

In fact the latter part of the lemma is easily deduced directly from
(\ref{firstthree}):
$$
Z(g) = g^{ij}g_ig_j+(g^ig_i)(g^j{}_j) = f^{ij}f_ig_j + (f^if_i)(g^j{}_j)\,,
$$
where, for a given $f$ with $\nabla f$ non-zero, the RHS is now a linear
operator on $g$, which, since the principal term is the Laplacian, is elliptic.

\begin{proof}  We have:
\begin{eqnarray*}
Z(f+\epsilon g) & = & Z(f) + \epsilon (g^{ij}f_if_j+2f^{ij}g_if_j+J\Delta g) \\
 & & \quad \epsilon^2(2g^{ij}g_if_j+f^{ij}g_ig_j+J\Delta f) + \epsilon^3Z(g) \\
 & = & Z(f) + \epsilon (f^{ij}f_ig_j+J\Delta g) + \epsilon^2(f^{ij}f_if_j
                                                +J\Delta f) + \epsilon^3Z(g) \\
 & = & (1+\epsilon^2)(Z(f) + \epsilon Z(g))\,.
\end{eqnarray*}
The last part of the lemma now follows from this formula, or as indicated
above, directly from (\ref{firstthree}).
\end{proof}

An interesting problem is to characterize those conjugate pairs that are
$3$-harmonic, i.e.~conjugate pairs $(f,g)$ satisfying $Z(f)=Z(g)=0$, for
then the mapping $(f,g)$ determines a $3$-harmonic morphism~\cite{Lo}. If both
$X$ and $Y$ vanish, then so does $Z$ and we have a complete description in this
case given by Theorem~\ref{thm:XYzero}. Up to conformal transformation, the
different conjugate $3$-harmonic pairs are given by
$$
\begin{array}{l}
(x_1, x_2), \quad 
\big(\frac12\log (x_1{}^2+x_2{}^2+x_3{}^2), \arctan(x_3/x_2)\big),\\
\qquad \quad \displaystyle 
\big(\frac{x_1}{x_1{}^2+x_2{}^2+x_3{}^2}, 
     \frac{x_2}{x_1{}^2+x_2{}^2+x_3{}^2}\big)\,.
\end{array}
$$
More generally, by the homogeneity of $Z(f)$ in $f$, the function $f$ is
$3$-harmonic if and only if it satisfies the linearisation of $Z$ at $f$:
${\mathcal Z}_f(f)=0$, so that by Lemma \ref{lem:Zfg}, ${\mathcal
Z}_f(f)={\mathcal Z}_f(g)=0$ is a necessary and sufficient condition for a
conjugate pair $(f,g)$ to be $3$-harmonic.

\appendix
\section{Conformal invariants}\label{conformalinvariants}
Suppose $f$ is a smooth function defined on an open subset
$U\subseteq{\mathbb{R}}^3$. As usual, we denote the partial derivatives of $f$
by subscripts 
$$f_i=\frac{\partial f}{\partial x^i},\quad f_{ij}=
\frac{\partial^2f}{\partial x^i\partial x^j},\quad f_{ijk}=
\frac{\partial^3f}{\partial x^i\partial x^j\partial x^k},\quad \ldots.$$
Equivalently, we may regard these quantities as tensors obtained by repeated
application of the flat connection $\nabla_i$ corresponding to the flat
metric~$\delta_{ij}$. Suppose $\Omega$ is a smooth non-vanishing function defined on
$U$ such that $\hat \delta_{ij}\equiv\Omega^2\delta_{ij}$ is also flat. If we let
$\Upsilon_i=\nabla_i\log\Omega$, then it is well-known~\cite{thebook} that
these functions are precisely the solutions of
$$\nabla_i\Upsilon_j=\Upsilon_i\Upsilon_j
-\textstyle\frac{1}{2}\delta_{ij}\Upsilon^k\Upsilon_k$$
and that all solutions are obtained by the conformal transformations of the 
round sphere $S^3$ viewed as flat-to-flat conformal rescalings via 
stereographic projection. Let $\hat\nabla_i$ denote the metric connection 
for $\hat \delta_{ij}$ and write 
$$\hat f=f,\quad \hat f_i=\hat\nabla_if,\quad 
\hat f_{ij}=\hat\nabla_i\hat\nabla_jf,\quad
\hat f_{ijk}=\hat\nabla_i\hat\nabla_j\hat\nabla_kf,\quad \ldots.$$ 
A conformal differential invariant of $f$ of weight $w$ is a polynomial
$$I=I(\delta^{ij},f,f_i,f_{ij},f_{ijk},\ldots)$$
in the derivatives of $f$ and the inverse metric $\delta^{ij}$ with the property
that it is invariant under arbitrary co\"ordinate transformation and 
\begin{equation}\label{defofinvariant}
I(\hat \delta^{ij},\hat f,\hat f_i,\hat f_{ij},\hat f_{ijk},\ldots)
=\Omega^wI(\delta^{ij},f,f_i,f_{ij},f_{ijk},\ldots)\end{equation} 
for all flat-to-flat conformal rescalings~$\Omega$. As detailed in~\cite{eg},
this notion of invariance is the same as requiring equivariance under the
action of ${\mathrm{SO}}(4,1)$ on the $3$-sphere, with
${\mathbb{R}}^3\hookrightarrow S^3$ by stereographic projection. It is
straightforward to write down explicit formulae for the effect of flat-to-flat 
rescalings on derivatives
\begin{equation}\label{jetchange}\begin{array}{rcl}\hat f_i&=&f_i\\
\hat f_{ij}&=&f_{ij}-2\Upsilon_{(i}f_{j)}+\delta_{ij}\Upsilon^kf_k\\
\hat f_{ijk}&=&f_{ijk}-6\Upsilon_{(i}f_{jk)}+3\delta_{(ij}\Upsilon^pf_{k)p}\\
&\vdots&\phantom{f_{ijk}}+6\Upsilon_{(i}\Upsilon_jf_{k)}
-3\delta_{(ij}\Upsilon_{k)}\Upsilon^pf_p-\frac32\Upsilon^p\Upsilon_p\delta_{(ij}f_{k)}
\end{array}\end{equation}
with a view to verifying (\ref{defofinvariant}) by direct calculation. It is
difficult to find conformal invariants from this direct point of view.
Certainly $J\equiv \delta^{ij}f_if_j=f^if_i$ is an invariant of weight $-2$. Perhaps
the simplest second order invariant is
$$Z\equiv f^{ij}f_if_j+Jf^j{}_j.$$
It has weight $-4$ but it is usual to omit the powers of $\Omega$ in verifying 
invariance (this is easily made precise by regarding the 
invariant as taking its values in an appropriate line-bundle). 
Specifically, as a linear combination of complete contractions it is 
manifestly invariant under co\"ordinate transformation and
$$\begin{array}{rcl}\hat f^{ij}\hat f_i\hat f_j&=&
f^{ij}f_if_j-\Upsilon^if_if^jf_j=f^{ij}f_if_j-J\Upsilon^kf_k\\
\hat J\hat f^j{}_j&=&Jf^j{}_j+J\Upsilon^kf_k
\end{array}$$
whence
$$\hat f^{ij}\hat f_i\hat f_j+\hat J\hat f^j{}_j=f^{ij}f_if_j+Jf^j{}_j,$$
as required. The familiar quantity 
\begin{equation}\label{defofX}
X=2f_i{}^jf_jf^{ik}f_k-f^if_if^{jk}f_{jk}+f^if_i(f^j{}_j)^2\end{equation}
is a conformal invariant of weight~$-6$. That it is a polynomial in the
derivatives of $f$ invariant under co\"ordinate change is already manifest. 
Its conformal invariance, however, is most easily seen from the identity of 
Lemma~\ref{magic}:--
$$JX+12T_{ijk}T^{ijk}=0,\quad\mbox{where }T_{ijk}=
f_{[i}\omega_jf_{k]l}\omega^l.$$ 
Here, recall that $\omega_j$ is any solution of the
equations~(\ref{firstthree}). We make take $\hat\omega_i=\omega_i$ to obtain a
solution of the conformally transformed equations resulting
from~(\ref{jetchange}). Then 
$$\hat f_{kl}\hat\omega^l
=f_{kl}\omega^l-\Upsilon^l\omega_lf_k+\Upsilon^lf_l\omega_k$$
and so $T_{ijk}$ is a conformally invariant tensor (of weight~$-2$). Notice,
however, that $T_{ijk}$ is not an expression solely in $f$ and its
derivatives but also involves~$\omega_j$. It may also be imaginary-valued. It
is only in the combination $T_{ijk}T^{ijk}$ that $\omega_j$ can be
eliminated using the relations~(\ref{firstthree}).
Of course, it is also possible to check the conformal invariance of $X$
directly from the expression~(\ref{defofX}).

In the remainder of this section we construct an extensive menagerie of
conformal differential invariants of~$f$. It is possible, in
principle~\cite{eg}, to list all such invariants. In practise, however, it is
easier to construct invariants by a number of tricks (see~\cite{s}). 
Apart from the particular invariant $V$ constructed below, these will
turn out to be sufficient for our purposes. The new connection $\hat\nabla_i$
is related to $\nabla_i$ by
$$\hat\nabla_i\phi_j=\nabla_i\phi_j-\Upsilon_i\phi_j-\Upsilon_j\phi_i
+\delta_{ij}\Upsilon^k\phi_k$$
when acting on an arbitrary $1$-form~$\phi_j$. It follows that 
$$\nabla^i[\Omega^{-1}\phi_i]=\Omega^{-1}\nabla^i\phi_i,$$
which we will more conveniently express by saying if $\phi_i$ has conformal
weight~$-1$, then $\phi_i\mapsto\nabla^i\phi_i$ is conformally invariant. 
Similarly, 
$$\textstyle\phi_j\mapsto\nabla_{(i}\phi_{j)}-\frac13\nabla^k\phi_k\delta_{ij}$$
is conformally invariant when $\phi_j$ has weight~$2$. Where $J$ does not
vanish we may consider the smooth $1$-form $J^{1/2}f_i$. It has weight $-1$ 
whence
$$J^{1/2}\nabla^j[J^{1/2}f_j]$$
is conformally invariant (of weight~$-4$). As written here, this is not a
polynomial but if we expand it we obtain
$$\textstyle \frac12[\nabla^jJ]f_j+J\nabla^jf_j=f^{ij}f_if_j+f^if_if^j{}_j\,,$$
which is a perfectly good polynomial. It follows that this is an invariant 
whether or not $J$ vanishes. It is our previous invariant~$Z$. Another 
viewpoint on this construction is that $f^j\nabla_jJ+2J\nabla_jf^j$ is a 
conformally invariant bilinear differential pairing between $f_i$ and~$J$. 
There are many such pairings on ${\mathbb{R}}^3$ as follows.
\begin{lemma}\label{firstpairings}
The following pairings are conformally invariant.
$$\begin{array}{ccl}
\underbrace{\psi}_{\mbox{\scriptsize\rm weight $v$}}\times
\underbrace{\phi}_{\mbox{\scriptsize\rm weight $w$}}&\mapsto&
\underbrace{v\psi\nabla_i\phi-w\phi\nabla_i\psi}_{
\mbox{\scriptsize\rm weight $v+w$}}\\
\rule{0pt}{16pt}\underbrace{\psi_i}_{\mbox{\scriptsize\rm weight $v$}}\times
\underbrace{\phi}_{\mbox{\scriptsize\rm weight $w$}}&\mapsto&
\underbrace{(v+1)\psi^i\nabla_i\phi-w\phi\nabla_i\psi^i}_{
\mbox{\scriptsize\rm scalar of weight $v+w-2$}}\\
\rule{0pt}{16pt}\mbox{ditto}&\mapsto&
\underbrace{v\psi_{[i}\nabla_{j]}\phi+w\phi\nabla_{[i}\psi_{j]}}_{
\mbox{\scriptsize\rm skew of weight $v+w$}}\\
\rule{0pt}{16pt}\mbox{ditto}&\mapsto&
(v-2)[\psi_{(i}\nabla_{j)}\phi-\frac13\delta_{ij}\psi^k\nabla_k\phi]\\
&&\underbrace{\rule{0pt}{12pt}\textstyle\qquad{}
-w\phi[\nabla_{(i}\psi_{j)}-\frac13\delta_{ij}\nabla^k\psi_k]}_{
\mbox{\scriptsize\rm symmetric trace-free of weight $v+w$}}\\
\end{array}$$
\end{lemma}
\begin{proof}
These are all easily verified by direct calculation. Alternatively, we may
employ evident variations on the trick used so far. For example, for
non-vanishing $\psi$ and $\phi$ we may write the first pairing as
$$\phi^{-v+1}\psi^{w+1}\nabla_i[\phi^v\psi^{-w}],$$
which is clearly invariant since $\phi^v\psi^{-w}$ has weight zero. All of
these pairings are similarly based on well-known conformally invariant linear
differential operators.
\end{proof}
Notice that the bundles occurring in these pairings are irreducible in the
sense that they are associated to irreducible representations of the orthogonal
group. These are the bundles between which it is relatively straightforward to
find invariant pairings. Here are two more examples that we shall need.
\begin{lemma}\label{secondpairings}
The following pairings are conformally invariant for $\psi$ of
weight $v$ and $\phi_{ij}$ being symmetric trace-free and of weight~$w$.
$$\begin{array}{ccl} \psi\times\phi_{ij}&\mapsto&
\underbrace{v\psi\nabla^i\phi_{ij}-(w+1)\phi_{ij}\nabla^i\psi}_{
\mbox{\scriptsize\rm weight $v+w-2$}}\\
\rule{0pt}{20pt}\psi\times\phi_{ij}&\mapsto&
v\psi[\nabla_{(i}\phi_{jk)}-\frac25\delta_{(ij}\nabla^l\phi_{k)l}]\\
&&\underbrace{\rule{0pt}{12pt}\textstyle\qquad{}
-(w-4)[\phi_{(ij}\nabla_{k)}\psi-\frac25\delta_{(ij}\phi_{k)l}\nabla^l\psi]}_{
\mbox{\scriptsize\rm symmetric trace-free of weight $v+w$}}
\end{array}$$
\end{lemma}
\begin{proof}
Easily verified by direct calculation.
\end{proof}
In fact, all the invariant pairings that we shall need may be constructed from
invariant linear differential operators. (There are, however, many invariant
pairings that do not arise in this way.) We are now able to list the
almost all the conformal invariants that we shall use.

\begin{thm}\label{usefullist}
The following are conformal differential invariants of a smooth function $f$ 
locally defined on ${\mathbb{R}}$. 
$$\begin{array}c
J\equiv f^if_i\qquad
Z\equiv f^{ij}f_if_j+Jf^j{}_j\\[3pt] 
X\equiv 2f_i{}^jf_jf^{ik}f_k-Jf^{jk}f_{jk}+J(f^j{}_j)^2\end{array}$$
If we now define
$$\begin{array}l\sigma_i\equiv J\nabla_iZ-2Z\nabla_iJ\qquad
\tau_i\equiv J\nabla_iX-3X\nabla_iJ\\[3pt]
\phi_{ij}\equiv \textstyle Jf_{ij}-2f_{(i}f_{j)}{}^kf_k
-\frac13Jf^k{}_k\delta_{ij}+\frac23f^{kl}f_kf_l\delta_{ij},
\end{array}$$
then the following are also conformal invariants.
$$\begin{array}c
R\equiv f^i\sigma_i              \qquad 
S\equiv f^i\tau_i                \qquad 
A\equiv \sigma^i\sigma_i         \qquad
B\equiv \tau^i\tau_i\\[3pt]
D\equiv \sigma^i\tau_i           \qquad
T\equiv \phi^{ij}\sigma_i\sigma_j\qquad
U\equiv \phi^{ij}\tau_i\tau_j
\end{array}$$
If we now define
$$\begin{array}l\textstyle\rho_{ijk}\equiv 
J\nabla_{(i}\phi_{jk)}-3\phi_{(ij}\nabla_{k)}J
-\frac25\delta_{(ij}\nabla^l\phi_{k)l}+\frac65\delta_{(ij}\phi_{k)l}\nabla^lJ\\[3pt]
\lambda_j\equiv 2J\nabla^i\phi_{ij}-\phi_{ij}\nabla^iJ,\end{array}$$
then the following are also conformal invariants.
$$\begin{array}c 
F\equiv\rho^{ijk}\phi_{ij}\lambda_k\qquad G\equiv\phi^{ij}\lambda_i\lambda_j
\qquad K\equiv \sigma^i\lambda_i\\[3pt]
M\equiv\tau^i\lambda_i\qquad N\equiv\sigma_i\rho^{ijk}\phi_{jk}
\qquad W\equiv\rho^{ijk}\rho_{ij}{}^l\phi_{kl}.\end{array}$$
\end{thm}
\begin{proof}
We have already observed that $J$, $Z$, and $X$ are conformally invariant. The
remaining invariants in this theorem are manufactured from these basic ones by
using Lemmata~\ref{firstpairings} and~\ref{secondpairings} as appropriate.
\end{proof}
There is one more invariant that we shall need and its construction is slightly
different. Let $Q^{ij}$ be any symmetric form and set 
$\upsilon = \epsilon^{jkl}(Jf_k{}{}^iQ_{ij} - f^iQ_{ij}f_{km}f^m)f_l$. Then the
following identity holds:
\begin{equation} \label{a-three}
 Y(Q^{ij}\omega_i\omega_j - Q^{ij}\eta_i\eta_j) = 4E\upsilon
\end{equation}
(recall that $E\equiv\epsilon^{ijk}f_i\omega_jf_k{}^\ell\omega_\ell$).
In the case when $Q^{ij} = f^{ijk}f_k - 2f^{ik}f_k{}^j$, one may check that
$\upsilon$ is conformally invariant. It is convenient and consistent 
with~\cite{be} to define the related
conformal invariant $V = 4J\upsilon$. It has a different character to our
previous invariants in that it changes sign under change of orientation. It is
said to be an odd invariant.

It is useful to record the conformal weight and homogeneity in $f$ for each of
the invariants of Theorem~\ref{usefullist} together with~$V$:--
\begin{center}
\begin{tabular}{r|c|c|c|c|c|c|c|c|c}
             &$J$ &$Z$ &$X$ &$R$ &$S$  &$V$  &$A$  &$B$  &$D$\\ \hline
\mbox{weight}&$-2$&$-4$&$-6$&$-8$&$-10$&$-11$&$-14$&$-18$&$-16$\\
\mbox{degree}&$2$ &$3$ &$4$ &$6$ &$7$  &$8$  &$10$ &$12$ &$11$
\end{tabular}\\
\rule{0pt}{8pt}\\
\begin{tabular}{r|c|c|c|c|c|c|c|c}
             &$T$  &$U$  &$F$  &$G$  &$K$  &$M$  &$N$  &$W$\\ \hline
\mbox{weight}&$-18$&$-22$&$-18$&$-18$&$-14$&$-16$&$-18$&$-18$\\
\mbox{degree}&$13$ &$15$ &$13$ &$13$ &$10$ &$11$ &$13$ &$13$
\end{tabular}
\end{center}
Any polynomial combination with consistent total weight will also be invariant.
For example, the quantity $Y=Z^2-2JX$ introduced in (\ref{thisisY}) is a
conformal invariant of weight~$-8$ (and homogeneity~$6$). Other evident
invariants are not necessarily new. For example, it is easily verified by 
direct computation that 
\begin{equation}\label{aha}\textstyle\phi^{ij}\phi_{ij}=\frac23Z^2-JX.
\end{equation}
This gives yet another verification that $X$ is conformally invariant.

\section{Invariant derivation of certain equations}  \label{sec:elim-om}
Our aim is to eliminate $\omega_i$ from polynomial expressions of the form
$F(f_i, f_{ij}, f_{ijk},...., \omega_i)$, given that the equations
(\ref{firstthree}) hold. We suppose that $X<0$, so that in particular $Y>0$.
Recall that
$$
\eta_i = \frac{1}{\sqrt{Y}}\left\{ 2(f^{kl}f_k\omega_l)f_i 
 + (Jf_k{}^k - f^{kl}f_kf_l)\omega_i - 2Jf_i{}^k\omega_k\right\}\,,
$$
gives the other conjugate direction, where an ambiguity of sign arises with the
choice of square root. 
\begin{lemma} Let $Q^{ij}$ be any symmetric form. Then
\begin{eqnarray}
 & &Y(Q^{ij}\omega_i\omega_j+ Q^{ij}\eta_i\eta_j)  = 
  2Q^{ij}f_if_j(JX-Z^2) \label{a-one} \\
 & & \qquad + 2J^2Q_j{}^j(Zf_l{}^l-X) 
  - 2J^2ZQ^{ij}f_{ij} +4JZQ^{ij}f_i{}^kf_kf_j \nonumber \\
& & \sqrt{Y}Q^{ij}\omega_i\eta_j  =  -ZQ^{ij}f_if_j + 2JQ^{ij}f_if_j{}^kf_k 
  + J^2(f_k{}^kQ_l{}^l - Q^{kl}f_{kl}) \label{a-two}
\end{eqnarray}
\end{lemma}
\begin{proof} Both formulae can be deduced by skew-symmetrising over the
indices of an appropriate $4$-tensor. For example, to derive the second,
consider the four tensor: $T_{ijkl} = f_i\omega_jf_k{}^kQ_l{}^l$ and apply the
identity: $T_{[ijkl]} = 0$. On transvecting first with $f^i$, then with
$\omega^j$ and applying (\ref{firstthree}), the result follows.
\end{proof}


Now let us find invariant proofs of some of the identities of
\S\ref{sec:resolution}. Recall
$$
\begin{array}{rcl}
p^+ & \equiv & f^{ijk}f_i\omega_j\omega_k + f^{ijk}f_if_jf_k 
            - 2 f^{ij}f_j{}^k\omega_i\omega_k + 2f^{ij}f_j{}^kf_if_k  \\
p^- & \equiv & f^{ijk}f_i\eta_j\eta_k + f^{ijk}f_if_jf_k 
                 - 2 f^{ij}f_j{}^k\eta_i\eta_k + 2f^{ij}f_j{}^kf_if_k \\
q^+ & \equiv & f^{ijk}\omega_i\omega_j\omega_k + f^{ijk}f_if_j\omega_k 
                                         + 4 f^{ij}f_j{}^kf_i\omega_k \\
q^- & \equiv & f^{ijk}\eta_i\eta_j\eta_k + f^{ijk}f_if_j\eta_k 
                                            + 4 f^{ij}f_j{}^kf_i\eta_k\,.
\end{array}
$$
\begin{thm} \label{thm:p-even}
The following identities hold:
\begin{eqnarray}
Y(p^++p^-) & = & ZS-2XR+2XY \label{p-even}\\
Y(p^+-p^-) & = &  EV/J\,. \label{p-odd}
\end{eqnarray}
where $X,Y,Z,R,S,V$ are the standard conformal invariants and where
$E\equiv\epsilon^{ijk}f_i\omega_jf_k{}^\ell\omega_\ell$.
\end{thm}
\begin{proof} The first identity is an application of (\ref{a-one}), where we
have set $Q^{ij} = f^{ijk}f_k - 2f^{ik}f_k{}^j$. For the second, we apply
(\ref{a-three}) with symmetric form $Q^{ij} = f^{ijk}f_k - 2f^{ik}f_k{}^j$.
\end{proof}
Note that both the LHS and the RHS of equation (\ref{p-odd}) change sign under
the interchange of the conjugate directions, the equation itself being
well-defined and independent of this operation.

The condition $p^+p^-\equiv0$ of Theorem \ref{mainthm} now follows since
$$4Y^2p^+p^-=Y^2(p^++p^-)^2 - Y^2(p^+-p^-)^2.$$ 
On applying (\ref{identityEX}),
this gives the necessary condition $P\equiv0$ of \S\ref{sec:resolution}:
$$
2(ZS-2XR+2XY)^2 + XV^2 = 0\,.
$$

Now consider the remaining conditions. We claim that we can use (\ref{a-one})
and (\ref{a-two}) to write $q^{ijk}\omega_i\omega_j\omega_k$ as a linear form
in $\omega_i$, where $q^{ijk}$ is any symmetric tensor.
 
For this, first set $Q^{ij} = q^{ijk}\omega_k$.  Then from (\ref{a-one}),
\begin{equation} \begin{array}{l}
Yq^{ijk}\omega_i\omega_j\omega_k  =  
-Y q^{ijk}\eta_i\eta_j\omega_k +2q^{ijk}f_if_j\omega_k(JX-Z^2) \\
\enskip{}+ 2J^2q_j{}^{jk}\omega_k(Zf_l{}^l-X) 
- 2J^2Zq^{ijk}f_{ij}\omega_k +4JZq^{ijk}f_i{}^lf_lf_j\omega_k\,. \label{a-four}
\end{array} \end{equation}
We now have to calculate $Y q^{ijk}\eta_i\eta_j\omega_k$. For this we set
$Q^{ij} = \sqrt{Y}q^{ijk}\eta_k$ and apply (\ref{a-two}):
$$ \begin{array}{l}
Yq^{ijk}\eta_i\eta_j\omega_k = \sqrt{Y}Q^{ij}\omega_i\eta_j \\  
= -Zq^{ijk}f_if_jv_k + 2Jq^{ijk}f_if_j{}^lf_lv_k 
                     + J^2(f_j{}^jq_l{}^{lk}v_k - q^{ijk}f_{ij}v_k)
\end{array}
$$
where $v_i=\sqrt{Y}\eta_i = 2(f^{kl}f_k\omega_l)f_i 
+ (Jf_k{}^k - f^{kl}f_kf_l)\omega_i - 2Jf_i{}^k\omega_k$. On expanding the 
right-hand side and substituting into (\ref{a-four}), we obtain:
$$  \begin{array}{l}
Yq^{ijk}\omega_i\omega_j\omega_k  =  
\omega_i\Big\{ q^{ijk}f_jf_k(-Y-2Zf^{lm}f_lf_m) \\
 + Jq^{ij}{}_j\Big[ Y+Z(f^{mn}f_mf_n) - 2(f^{mn}f_mf_n)^2\Big] \\
 +J(Z+2f^{lm}f_lf_m)(2q^{ijk}f_j{}^lf_lf_k - Jq^{ijk}f_{jk}) 
                                           - 2JZq^{jkl}f_jf_kf_l{}^i \\
 -2f^{in}f_n\Big[  2Jq^{jkl}f_jf_kf_l{}^mf_m + J^2f_j{}^jq_k{}^{kl}f_l 
                       - J^2q^{jkl}f_{jk}f_l - Zq^{jkl}f_jf_kf_l\Big] \\
 + 4J^2q^{jkl}f_jf_k{}^mf_mf_l{}^i + 2J^3(f_j{}^jq_k{}^{kl}f_l{}^i 
                                   - q^{jkl}f_{jk}f_l{}^i) \Big\}\,, 
\end{array}
$$
as claimed.
 
We can now express $Yq^+$ by setting $q^{ijk} = f^{ijk}$ and then adding 
$Y(f^{ijk}f_jf_k\omega_i + 4f^{jk}f_k{}^if_j\omega_i)$:
$$
\begin{array}{l} Yq^+  =   
\omega_i\Big\{ -2Zf^{ijk}f_jf_k(f^{lm}f_lf_m) \\ 
 +Jf^{ij}{}_j\Big[ Y+Z(f^{mn}f_mf_n) - 2(f^{mn}f_mf_n)^2\Big] \\
 +J(Z+2f^{lm}f_lf_m)(2f^{ijk}f_j{}^lf_lf_k - Jf^{ijk}f_{jk}) 
 - 2JZq^{jkl}f_jf_kf_l{}^i \\
 -2f^{in}f_n\Big[ 2Jf^{jkl}f_jf_kf_l{}^mf_m + J^2f_j{}^jf_k{}^{kl}f_l 
 - J^2f^{jkl}f_{jk}f_l - Zf^{jkl}f_jf_kf_l\Big] \\
 + 4J^2f^{jkl}f_jf_k{}^mf_mf_l{}^i + 2J^3(f_j{}^jf_k{}^{kl}f_l{}^i 
 - f^{jkl}f_{jk}f_l{}^i) + 4f^{jk}f_k{}^if_j\Big\} 
\end{array}
$$
This has the form $Yq^+ \equiv \alpha^i\omega_i$, where each $\alpha^i$ is an
explicit Riemannian invariant polynomial expression in $f_i, f_{ij}, f_{ijk}$,
which at each point is defined up to addition of an arbitrary linear
combination:
$$
af^i + b[(f^{kl}f_kf_l)\omega^i + Jf^{ik}\omega_k]\,.
$$
By symmetry, we must also have $Yq^- = \alpha^i\eta_i$. Then the fifth
condition $Y\sqrt{Y}q^+q^-\equiv 0$ has the form $r^{ij}\omega_i\eta_j = 0$,
where $r^{ij}$ is the symmetric form $r^{ij} = (1/\sqrt{Y})\alpha^i\alpha^j$.
We can now apply (\ref{a-two}) to give an invariant derivation of the quantity
$Q\equiv Y\sqrt{Y}q^+q^-$ of~\S\ref{sec:resolution}.
 
The final equation of Theorem \ref{mainthm} is $(p^+q^-)^2 + (p^-q^+)^2 = 0$.
But
$$
\begin{array}{l}
4\{ (p^+q^-)^2 + (p^-q^+)^2\} = 
\{ (p^++p^-)^2 + (p^+-p^-)^2\} \{ (q^+)^2 + (q^-)^2\} \\  
- 2(p^++p^-)(p^+-p^-)\{ (q^+)^2 -(q^-)^2\}\,,
\end{array}
$$
which we can see as a product of conformally invariant terms that we can deal
with. First, multiply the whole expression by $Y^3\sqrt{Y}$. Then $Y(p^++p^-)$
is given by (\ref{p-even}), whilst $Y(p^+-p^-)$ is given by (\ref{p-odd}). On
the other hand, 
$$Y\sqrt{Y}((q^+)^2 + (q^-)^2) = r^{ij}\omega_i\omega_j +r^{ij}\eta_i\eta_j,$$
which can be expressed using (\ref{a-one}) above, whereas
$$Y\sqrt{Y}((q^+)^2 - (q^-)^2) = r^{ij}\omega_i\omega_j-r^{ij}\eta_i\eta_j$$
can be expressed using (\ref{a-three}). Note that the result involves $E^2$,
which by (\ref{identityEX}) can be written in terms of conformal invariants 
of~\S\ref{conformalinvariants}.

\section{Normalising conformal Killing fields} \label{normalised_conf_Killing}
The conformal Killing fields on ${\mathbb{R}}^3$ form a finite-dimensional
vector space on which ${\mathrm{O}}(4,1)$ acts via the conformal automorphisms
of~$S^3$. It is the adjoint representation~${\mathfrak{o}}(4,1)$ and so the
question of normalising a conformal Killing field up to conformal
transformations comes down to finding canonical representatives for the orbits
of this action. This is a question of linear algebra, which may be stated more
generally as follows. Suppose we are given a real symmetric $n\times n$ matrix
$H$ of Lorentzian signature meaning that there is a real invertible $n\times n$
matrix such that
\begin{equation}\label{Lorentz}
A^tHA=\mbox{\small$\left[\begin{array}{cccc}1&0&0&0\\ 
0&\raisebox{-1pt}[0pt]{$\ddots$}&0&0\\
0&0&1&0\\ 0&0&0&-1
\end{array}\right]$}.\end{equation}
Suppose $N$ is a real skew $n\times n$ matrix. We would like to find a real
invertible $n\times n$ matrix $A$ such that $A^tHA$ and $A^tNA$ are placed in
some canonical form. For example, we may insist on (\ref{Lorentz}) for $A^tHA$ 
but following \cite{laplace,eg} we normally prefer (written in block form) 
\begin{equation}\label{prefer}
A^tHA=\mbox{\small$\left[\begin{array}{ccc}0&0&1\\
0&{\mathrm{Id}}&0\\ 1&0&0
\end{array}\right]$},\end{equation}
where ${\mathrm{Id}}$ is the $(n-2)\times (n-2)$ identity matrix.
\begin{lemma}\label{onlyone}
Suppose $H$ is a real symmetric $n\times n$ matrix of Lorentzian
signature and $N$ is a real skew $n\times n$ matrix. Suppose that, regarded as
a complex matrix, $H^{-1}N$ has only one eigenvector up to scale. Then, the
eigenvalue is zero, it must be that $n=3$, and we can find an invertible real
$3\times 3$ matrix $A$ such that
$$A^tHA=\mbox{\small$\left[\begin{array}{ccc}0&0&1\\
0&1&0\\ 1&0&0
\end{array}\right]$}\quad\mbox{and}\quad
A^tNA=\mbox{\small$\left[\begin{array}{ccc}0&0&0\\ 
0&0&2\\ 0&-2&0
\end{array}\right]$}.$$
\end{lemma}
\begin{proof} Notice that
$$H\mapsto A^tHA\quad\mbox{and}\quad N\mapsto A^tNA\quad\implies
H^{-1}N\mapsto A^{-1}H^{-1}NA.$$
Therefore, without loss of generality, we may suppose that $H^{-1}N$ is in
Jordan canonical form. Our hypothesis says that there is just one Jordan block
with the eigenvalue $\lambda$ down the diagonal. But
$$\tr(H^{-1}N)=\tr(N^t(H^t)^{-1})=
-\tr(NH^{-1})=-\tr(H^{-1}N)$$
so $\lambda=0$. In particular, the eigenspace is the same as the kernel of~$N$.
Suppose $u$ is a non-zero vector in this kernel and consider
$$u^\perp\equiv\{v\mbox{ s.t. }u^tHv=0\}.$$
Since 
$$u^tHH^{-1}Nv=u^tNv=v^tN^tu=-v^tNu=0,$$
it follows that $H^{-1}N$ preserves $u^\perp$. The hypothesis that $H^{-1}N$
has only one eigenvector up to scale now forces $u\in u^\perp$. In other words
$u$ is null, i.e.~$u^tHu=0$. It is well-known that ${\mathrm{O}}(n-1,1)$ acts
transitively on the null vectors. Therefore we may suppose that
$$H=\mbox{\small$\left[\begin{array}{ccc}0&0&1\\
0&{\mathrm{Id}}&0\\ 1&0&0
\end{array}\right]$}\quad\mbox{and}\quad
u=\mbox{\small$\left[\begin{array}{c}1\\ 0\\ 0
\end{array}\right]$}.$$
It follows that 
$$N=\mbox{\small$\left[\begin{array}{ccc}0&0&0\\ 
0&M&-r\\ 0&r^t&0\end{array}\right]$}$$
where $M$ is a skew $(n-2)\times(n-2)$ matrix. Therefore, 
$$H^{-1}N=\mbox{\small$\left[\begin{array}{ccc}0&r^t&0\\ 
0&M&-r\\ 0&0&0\end{array}\right]$}\quad\mbox{and}\quad
(H^{-1}N)^2=\mbox{\small$\left[\begin{array}{ccc}0&r^tM&-r^tr\\ 
0&M^2&-Mr\\ 0&0&0\end{array}\right]$}$$
{From} the Jordan canonical form of $H^{-1}N$ we see that, not only does its
trace vanish, but also the traces of its higher powers. In particular,  
$$0=\tr((H^{-1}N)^2) = \tr(M^2)$$
and since $M$ is skew it follows that $M=0$ and hence that $\rk N=2$. Since
the kernel of $N$ is supposedly $1$-dimensional, $n=3$ is forced and
$$N=\mbox{\small$\left[\begin{array}{ccc}0&0&0\\ 
0&0&-r\\ 0&r&0\end{array}\right]$}.$$
Finally, if we take
$$A=\mbox{\small$\left[\begin{array}{ccc}\mu^{-1}&0&0\\ 
0&1&0\\ 0&0&\mu\end{array}\right]$},$$
then 
$$A^tHA=\mbox{\small$\left[\begin{array}{ccc}0&0&1\\ 
0&1&0\\ 1&0&0
\end{array}\right]$}\quad\mbox{and}\quad
A^tNA=\mbox{\small$\left[\begin{array}{ccc}0&0&0\\ 
0&0&-\mu r\\ 0&\mu r&0
\end{array}\right]$}$$
and so we can insist that $\mu r=-2$ if we so wish.  
\end{proof}
\begin{lemma}\label{ontheaxes}
Suppose $H$ is a real symmetric $n\times n$ matrix of Lorentzian signature and
$N$ is a real skew $n\times n$ matrix. Then the eigenvalues of $H^{-1}N$ lie on
the real or imaginary axes.
\end{lemma}
\begin{proof}Suppose that $x+iy$ is an eigenvalue, i.e.\
\begin{equation}\label{eigenvalue}
H^{-1}N(u+iv)=(x+iy)(u+iv)\quad\mbox{for some }u+iv\not=0.
\end{equation}
Writing out the real and imaginary parts separately gives
\begin{equation}\label{realandimaginary}
H^{-1}Nu=x u-y v\quad\mbox{and}
\quad H^{-1}Nv=y u+x v.\end{equation}
We argue by contradiction, supposing that neither $x$ nor $y$
vanishes. In this case we see from (\ref{realandimaginary}) that neither $u$
nor $v$ vanishes. Because $N$ is skew, we see from (\ref{realandimaginary})
that 
$$\begin{array}l
0=u^tNu=u^tHH^{-1}Nu=x u^tHu-y u^tHv,\\
0=v^tNv=v^tHH^{-1}Nv=y v^tHu+x v^tHv.
\end{array}$$
Therefore
$$x u^tHu=y u^tHv=y v^tHu=-x v^tHv.$$
Since we are supposing that $x\not=0$, we conclude that
$u^tHu=-v^tHv$. Again using~(\ref{realandimaginary}), we now find that 
$$0=u^tNv+v^tNu=u^tH^{-1}HNv+v^tH^{-1}HNu=2y u^tHu+2x u^tHv,$$
whence
$$0=x u^tHu-y u^tHv\quad\mbox{and}\quad
0=y u^tHu+x u^tHv.$$
Therefore $(x^2+y^2)u^tHu=0$ and so $u^tHu=0$. Bearing in mind our assumption
that $y\not=0$, we have found two real vectors $u$ and $v$ with
$$u\not =0,\quad v\not=0,\quad u^tHu=0,\quad v^tHv=0,\quad u^tHv=0.$$
For $H$ of Lorentzian signature this forces $v=tu$ for some $t\in{\mathbb{R}}$.
Substituting back into (\ref{eigenvalue}) and taking out a factor of $(1+it)$
gives 
$$H^{-1}Nu=(x+iy)u$$
and hence that $y=0$, our required contradiction.
\end{proof}
\begin{lemma}\label{realnonzero}
Suppose $H$ is a real symmetric $n\times n$ matrix of Lorentzian signature and
$N$ is a real skew $n\times n$ matrix. Suppose that $H^{-1}N$ has a non-zero
real eigenvalue~$\lambda$. Then $-\lambda$ is also an eigenvalue and we can
find an invertible real $n\times n$ matrix $A$ such that (in block form)
$$A^tHA=\mbox{\small$\left[\begin{array}{ccc}0&0&1\\
0&{\mathrm{Id}}&0\\ 1&0&0
\end{array}\right]$}\quad\mbox{and}\quad
A^tNA=\mbox{\small$\left[\begin{array}{ccc}0&0&\lambda\\ 
0&M&0\\ -\lambda&0&0
\end{array}\right]$},$$
where $M$ is a skew $(n-2)\times(n-2)$ matrix.
\end{lemma}
\begin{proof} Certainly, we may arrange that $A^tHA$ is of the required
form and we shall suppose, without loss of generality, that $H$ is already
normalised like this. Write $H^{-1}Nu=\lambda u$ for $u\not=0$. Then 
$$0=u^tNu=u^tHH^{-1}Nu=\lambda u^tHu$$
so $u^tHu=0$. Therefore, by a suitable $A$ we may arrange  
$$u=\mbox{\small$\left[\begin{array}c0\\ 0\\ 1
\end{array}\right]$},\enskip\mbox{ and this forces }\enskip
H^{-1}N=\mbox{\small$\left[\begin{array}{ccc}\cdot&\cdot&0\\
\cdot&\cdot&0\\ \cdot&\cdot&\lambda
\end{array}\right]$}.$$
Bearing in the mind that $N$ is skew, this implies 
$$N=\mbox{\small$\left[\begin{array}{ccc}\cdot&\cdot&\lambda\\
\cdot&\cdot&0\\ -\lambda&0&0
\end{array}\right]$},\enskip\mbox{ and then }
H^{-1}N=\mbox{\small$\left[\begin{array}{ccc}-\lambda&0&0\\
\cdot&\cdot&0\\ \cdot&\cdot&\lambda
\end{array}\right]$}.$$
It follows that $-\lambda$ is an eigenvalue, say $H^{-1}Nv=-\lambda v$ for some
$v\not=0$ and, reasoning as above, $v^tHv=0$. Since $u$ and $v$ are not
proportional, we may scale them so that $u^tHv=1$. Finally, if we arrange 
that  
$$v=\mbox{\small$\left[\begin{array}c1\\ 0\\ 0
\end{array}\right]$}
,\enskip\mbox{ then}\enskip
H^{-1}N=\mbox{\small$\left[\begin{array}{ccc}-\lambda&0&0\\
0&\cdot&0\\ 0&\cdot&\lambda
\end{array}\right]$}.$$
This immediately implies that $N$ has the desired form.
\end{proof}

With these Lemmata on hand we are now in a position to establish a general
canonical form. As already mentioned, we shall prefer (\ref{prefer}) for
$A^tHA$. When $n=2$ there is almost nothing more to do:--
$$A^tHA=\mbox{\small$\left[\begin{array}{ccc}0&1\\ 1&0
\end{array}\right]$}\quad\mbox{and}\quad
A^tNA=\mbox{\small$\left[\begin{array}{ccc}0&\lambda\\ -\lambda&0
\end{array}\right]$}$$
simply because $N$ is skew. It remains to observe that we can change the sign
of $\lambda$ using
$$A=\mbox{\small$\left[\begin{array}{ccc}0&1\\ 1&0
\end{array}\right]$}$$
but that $\lambda^2$ is well-defined because the characteristic polynomial  
$$\det(H^{-1}N-t\,{\mathrm{id}})=t^2-\lambda^2$$
is invariant. The first interesting case is $n=3$. 
\begin{thm}
Suppose $H$ is a real symmetric $3\times 3$ matrix of Lorentzian
signature and $N$ is a real skew $3\times 3$ matrix. Then we can find an
invertible real $3\times 3$ matrix $A$ such that
$$A^tHA=\mbox{\small$\left[\begin{array}{ccc}0&0&1\\
0&1&0\\ 1&0&0
\end{array}\right]$}$$
and $A^tNA$ is
\begin{equation}\label{threecases}
\mbox{\small$\left[\begin{array}{ccc}0&0&\lambda\\ 
0&0&0\\ -\lambda&0&0
\end{array}\right]$}\quad\mbox{or}\quad
\frac{1}{\sqrt 2}\mbox{\small$\left[\begin{array}{ccc}0&\lambda&0\\ 
-\lambda&0&-\lambda\\ 0&\lambda&0
\end{array}\right]$}\quad\mbox{or}\quad
\mbox{\small$\left[\begin{array}{ccc}0&0&0\\ 
0&0&2\\ 0&-2&0
\end{array}\right]$}.\end{equation}
Furthermore, these three possible canonical forms are distinct apart from
changing the sign of $\lambda$ in the first two cases and the coincidence of
the first two cases when $\lambda=0$.
\end{thm}
\begin{proof}
If $H^{-1}N$ has only one eigenvector up to scale, then Lemma~\ref{onlyone}
applies and we obtain the third case of~(\ref{threecases}). Else,
Lemma~\ref{ontheaxes} implies that either all eigenvalues are real or they are
$i\lambda,-i\lambda,0$ for some $\lambda\not=0$.

Firstly, let us suppose they are all real. They could still all be zero in
which case the kernel $N$ is at least $2$-dimensional. But the rank of a skew
matrix is always even so then $N=0$. Otherwise, if $\lambda\not=0$ is a
real eigenvalue, then Lemma~\ref{realnonzero} gives the first
of~(\ref{threecases}). 

When $i\lambda$ is an eigenvalue, then 
$$H^{-1}N(u+iv)=i\lambda(u+iv)$$
implies that 
$$H^{-1}Nu=-\lambda v\quad\mbox{and}\quad H^{-1}Nv=\lambda u.$$
It follows that
$$\begin{array}l0=u^tNu=u^tHH^{-1}Nu=-\lambda u^tHv\\
0=u^tNv+v^tNu=u^tHH^{-1}Nv+v^tHH^{-1}Nu=\lambda(u^tHu-v^tHv)
\end{array}$$
and so if $\lambda\not=0$, we conclude that 
$$u^tHu=v^tHv\quad\mbox{and}\quad u^tHv=0.$$
In this case, by a suitable $A$ we may arrange 
$$u=\frac{1}{\sqrt 2}
\mbox{\small$\left[\begin{array}c1\\ 0\\ 1\end{array}\right]$}
\quad\mbox{and}\quad
v=\mbox{\small$\left[\begin{array}c0\\ 1\\ 0\end{array}\right]$},$$
from which the second of (\ref{threecases}) is forced. Interchanging $u$ and
$v$ changes the sign of~$\lambda$. Otherwise, the distinctions between these
canonical forms is clear from the Jordan canonical form of $H^{-1}N$ and its
characteristic polynomial.
\end{proof}
It is easy to generalise these canonical forms to $n\times n$ matrices. The
only one we shall need is the $5\times 5$ case and we state it here. 
\begin{thm}\label{fivebyfivecase}
Suppose $H$ is a real symmetric $5\times 5$ matrix of Lorentzian
signature and $N$ is a real skew $5\times 5$ matrix. Then we can find an
invertible real $5\times 5$ matrix $A$ such that
$$A^tHA=\mbox{\small$\left[\begin{array}{ccccc}0&0&0&0&1\\
0&1&0&0&0\\ 0&0&1&0&0\\ 0&0&0&1&0\\ 1&0&0&0&0
\end{array}\right]$}$$
and $A^tNA$ is
\begin{equation}\label{firsttype}
\mbox{\small$\left[\begin{array}{ccccc}0&0&0&0&\lambda\\ 
0&0&0&0&0\\ 0&0&0&\mu&0\\ 0&0&-\mu&0&0\\ -\lambda&0&0&0&0
\end{array}\right]$}\quad
\begin{array}l\mbox{well-defined up to}\\
(\lambda,\mu)\mapsto(-\lambda,\mu)\mbox{ or }(\lambda,-\mu)\\
\mbox{or }(-\lambda,-\mu),\end{array}\end{equation}
or
$$\mbox{\small$\left[\begin{array}{ccccc}0&\lambda/\sqrt 2&0&0&0\\ 
-\lambda/\sqrt 2&0&0&0&-\lambda/\sqrt 2\\ 
0&0&0&\mu&0\\ 0&0&-\mu&0&0\\ 0&\lambda/\sqrt 2&0&0&0
\end{array}\right]$}\quad
\begin{array}l\mbox{well-defined up to}\\
(\lambda,\mu)\mapsto(-\lambda,\mu)\mbox{ or }(\lambda,-\mu)\\
\mbox{or }(-\lambda,-\mu)\mbox{ or }(\mu,\lambda)\\ 
\mbox{or }(-\mu,\lambda)\mbox{ or }(\mu,-\lambda)\\
\mbox{or }(-\mu,-\lambda),
\end{array}$$
or
\begin{equation}\label{thirdtype}
\mbox{\small$\left[\begin{array}{ccccc}0&0&0&0&0\\ 
0&0&0&0&2\\ 0&0&0&\mu&0\\ 0&0&-\mu&0&0\\ 0&-2&0&0&0
\end{array}\right]$}
\quad\mbox{well-defined up to }\mu\mapsto-\mu.\end{equation}
Furthermore, these canonical forms are distinct except for the evident
coincidence of the first two cases when~$\lambda=0$.
\end{thm}

\end{document}